\documentclass[a4paper,11pt, dvipsnames]{scrartcl}
\usepackage{etex}
\addtokomafont{disposition}{\rmfamily}

\usepackage{amsthm}
\usepackage{amsmath}
\usepackage{amssymb}
\usepackage{mathrsfs} 
\usepackage{graphicx}
\usepackage{xcolor}
\usepackage{multirow}
\usepackage[
  natbibapa
]{apacite} 

\usepackage[colorlinks=true,linkcolor=NavyBlue, citecolor=NavyBlue, urlcolor=NavyBlue, linktocpage=true]{hyperref} 
\usepackage{bbm}
\usepackage{textcomp}
\usepackage{enumerate}
\usepackage{mathtools}

\usepackage{framed}
\usepackage{comment}
\definecolor{shadecolor}{gray}{0.9}

\usepackage{dsfont} 
\usepackage{caption} 
\usepackage{subcaption} 
\usepackage{graphicx}
\usepackage{pdfpages}
\usepackage{url}
\usepackage[english]{babel}

\specialcomment{extra}{\begin{shaded}}{\end{shaded}}

\theoremstyle{plain}  
\newtheorem{thm}{Theorem}[section] 
\newtheorem{lem}[thm]{Lemma} 
\newtheorem{prop}[thm]{Proposition} 
\newtheorem{cor}[thm]{Corollary} 

\theoremstyle{definition} 
\newtheorem{defn}[thm]{Definition}
 
\newtheorem{exmp}[thm]{Example}
\newtheorem{rem}[thm]{Remark}

\newtheoremstyle{assumption}
{3pt}
{3pt}
{}
{}
{\bf}
{.}
{.5em}
{\thmname{#1} (\thmnote{#3}\thmnumber{#2})}

\theoremstyle{assumption}

\theoremstyle{remark}

\newcommand{\diff}{\mathrm{d}}
\newcommand{\dint}{\,\mathrm{d}}
\newcommand{\E}{\mathbf{E}}

\newcommand{\eps}{\varepsilon}
\renewcommand{\l}{\lambda}

\DeclareMathOperator*{\essinf}{ess\,inf}
\DeclareMathOperator*{\esssup}{ess\,sup}

\newcommand{\F}{\mathcal{F}}
\newcommand{\R}{\mathbb{R}}
\newcommand{\A}{\mathsf{A}}
\renewcommand{\O}{\mathsf{O}}
\newcommand{\one}{\mathds{1}}
\newcommand{\interior}{\operatorname{int}}

\DeclareMathOperator*{\argmin}{arg\,min}

\newcommand{\N}{\mathbb N}

\renewcommand{\a}{\alpha}

\newcommand{\M}{\mathcal M}

\newcommand{\T}{T}

\newcommand{\bY}{\mathbf Y}

\newcommand{\As}{{\mathsf{A}_{\text{\rm sel}}}}
\newcommand{\Ae}{{\mathsf{A}_{\text{\rm exh}}}}
\newcommand{\Ss}{S_{\text{\rm sel}}}
\newcommand{\Se}{S_{\text{\rm exh}}}
\newcommand{\Vs}{V_{\text{\rm sel}}}

\newcommand{\cl}{\mathrm{cl}}

\renewcommand{\S}{\mathcal S}

\newcommand{\bP}{\mathbf{P}}

\newcommand{\I}{\mathcal I}
\newcommand{\C}{\mathcal C}

\DeclareMathOperator*{\graph}{graph}

\DeclareMathOperator*{\epi}{epi}

\renewcommand{\rm}{\normalfont \rmfamily}
\renewcommand{\bf}{\normalfont \bfseries}

\def\be{\begin{equation} \label}
\def\ee{\end{equation}}

\numberwithin{equation}{section}

\usepackage[colorinlistoftodos,textsize=tiny]{todonotes}
\newcommand{\Comments}{1}
\newcommand{\mynote}[2]{\ifnum\Comments=1\textcolor{#1}{#2}\fi}
\newcommand{\mytodo}[2]{\ifnum\Comments=1
  \todo[linecolor=#1!80!black,backgroundcolor=#1,bordercolor=#1!80!black]{#2}\fi}

\ifnum\Comments=1               
\fi

\newcommand{\tobi}[1]{\mytodo{blue!20!white}{TF: #1}}
\ifnum\Comments=1               
\fi

\ifnum\Comments=1               
\fi

\ifnum\Comments=1               
\fi

\begin{document}

\title{Forecast Evaluation of Quantiles, Prediction Intervals, and other Set-Valued Functionals}
\author{Tobias Fissler\thanks{Vienna University of Economics and Business, Institute for Statistics and Mathematics, Welthandels\-platz 1, 1020 Vienna, Austria, 
\texttt{tobias.fissler@wu.ac.at},
\texttt{jana.hlavinova@wu.ac.at} and \texttt{birgit.rudloff@wu.ac.at}}
\and Rafael Frongillo\thanks{University of Colorado Boulder, Department of Computer Science, 430 UCB, 1111 Engineering Dr, Boulder, CO 80309, \texttt{raf@colorado.edu}}
\and Jana Hlavinov\'{a}\footnotemark[1]
\and Birgit Rudloff\footnotemark[1]
}
\maketitle

\begin{abstract}
\noindent
\textbf{Abstract.}
We introduce a theoretical framework of elicitability and identifiability of set-valued functionals, such as quantiles, prediction intervals, and systemic risk measures.
A functional is elicitable if it is the unique minimiser of an expected scoring function, and identifiable if it is the unique zero of an expected identification function; both notions are essential for forecast ranking and validation, and $M$- and $Z$-estimation. 
Our framework distinguishes between \emph{exhaustive} forecasts, being set-valued and aiming at correctly specifying the entire functional, and \emph{selective} forecasts, content with solely specifying a single point in the correct functional.  
We establish a mutual exclusivity result: A set-valued functional can be either selectively elicitable or exhaustively elicitable or not elicitable at all. 
Notably, since quantiles are well known to be selectively elicitable, they fail to be exhaustively elicitable. 
We further show that the class of prediction intervals and Vorob'ev quantiles turn out to be exhaustively elicitable and selectively identifiable.
In particular, we provide a mixture representation of elementary exhaustive scores, leading the way to Murphy diagrams.
We give possibility and impossibility results for the shortest prediction interval and prediction intervals specified by an endpoint or a midpoint.
We end with a comprehensive literature review on common practice in forecast evaluation of set-valued functionals.
\end{abstract}



\section{Introduction}

Humanity faces the need to make decisions despite uncertainty.
This uncertainty has many sources, one of which stems from unknown or random future events.
For example, in agriculture the right time for harvesting depends on the weather in the forthcoming days; in business, an investment decision in a production facility depends on future demand; in politics, decisions for travel restrictions or lockdowns heavily depend on the anticipated future development of a pandemic.
Predicting uncertain future events is therefore urgent and ubiquitous, dating back at least to the ancient Delphic Oracle and spanning the time up to today's sophisticated quantitative epidemiological models.

The presence of various sorts of forecasts and humanity's reliance on them calls for a careful assessment and evaluation, basically focusing on two complementary aspects:
First, are given forecasts good or reliable in absolute terms?
And second, how well have certain forecasts performed in comparison to some alternative predictions, assessing their relative quality? 
Clearly, these questions can only be answered \emph{ex post}, given observations of the future events in question.
Then, the reliability or calibration can be assessed in terms of \emph{moment} or \emph{identification functions}.
Forecast comparison and ranking, in turn, is commonly performed in terms of loss or scoring functions.
To perform forecast evaluation properly, one must specify a certain quality criterion or directive the forecasts should follow. This directive might be given indirectly in terms of a cost, loss, or score, such that good forecasts aim to minimise this criterion (in expectation).
This directive can also be formulated directly, such as the whole probability distribution of the uncertain event, capturing the entire inherent uncertainty, or as a summary measure thereof, called \emph{functional}, such as the mean, variance, or a certain risk measure.
When the directive is specified in the latter way,
it is crucial that the tools of forecast evaluation, chiefly scoring and identification functions, are in line with this directive.
This alignment leads to the notions of strictly consistent scores, which are minimised in expectation by the correctly specified forecasts, and strict identification functions, whose roots in expectation are the correctly specified forecasts. 

The literature on forecast evaluation has mainly focused on single-valued functionals $T$, such as real-valued and vector-valued point forecasts or probabilistic forecasts, where a single correct value is specified for each possible distribution (for technical definitions and an account of the literature, we refer to Section \ref{subsec:consistent scoring functions}).
Yet set-valued functionals are abound; prominent examples include quantiles,\footnote{This coincides with option (ii) of the interesting discussion provided in \citet[p.\ 170]{Mizera2010}.} the mode, and prediction intervals, which may all be non-unique and therefore set-valued.
Applications bring many other examples, such as the set-valued systemic risk measures introduced in \cite{FeinsteinRudloffWeber2017} which specify the entire set of capital allocations adequate to render a financial system's risk acceptable.
Set-valued functionals also naturally arise via expectations or quantiles of random sets \citep{Molchanov2017}, such as from climatology and meteorology (the area affected by a flood), reliability engineering (parts of a machine being affected by extreme heat) or medicine (tumorous tissue in the human body); cf.\ \cite{BolinLindgren2015}.
We discuss these applications and several others in Section~\ref{subsec:literature} with a special focus on quantiles of random sets in Section~\ref{sec:Vorob'ev Quantiles}.

Techniques developed to evaluate single-valued forecasts do not in general suffice for set-valued functionals.
For functionals like the mode or the $\alpha$-quantile, it is common to restrict to the set of distributions with a unique mode or $\alpha$-quantile~\citep{Heinrich2014, FisslerZiegel2019}, yet one may be interested in distributions with multiple modes or quantiles.
Moreover, many functionals, such as quantiles of random sets, are inherently set-valued and such a restriction is not available, or too much of a simplification.
We therefore need a comprehensive theoretical framework for the assessment of forecasts for set-valued functionals.

It turns out that already for the definition of elicitation (or identification) of a set-valued functional, there are several possibilities, depending on the form the forecasts take.
One may ask for an arbitrary element of the functional, an arbitrary subset, or perhaps even the entire set itself.
For the case of the $\alpha$-prediction interval, and the uniform distribution on $[0,1]$, these definitions correspond to any subinterval of $[0,1]$ of length at least $\alpha$, any set of such intervals, or the entire set of all such intervals.
Moreover, if a set-valued functional is elicitable in one of these corresponding senses, does it continue to be elicitable if one specifies a particular element of the functional to be elicited, such as the shortest $\alpha$-prediction interval; or if it is not elicitable in one of the senses above, could such a specification render it elicitable? 

In this paper, we present a general theoretical framework to evaluate the forecasts of set-valued functionals, which clarifies and expands upon these questions.
We begin with a thorough definition of elicitability and identifiability of set-valued functionals (Section~\ref{sec:theory}).
In particular, as alluded to in the above questions, we define two types of set-valued elicitation (identification):
for the \emph{selective} type, we follow \cite{LambertShoham2009} and \cite{Gneiting2011} where a single-valued forecast must be among the set of correct values specified by the functional---as it is also typical for quantiles \citep{Koenker2005}, whereas the \emph{exhaustive} type more ambitiously asks one to forecast the entire set of correct values, and requires this set to be the unique minimiser (zero) of the expected score (identification function).

The main result of this article, Theorem~\ref{thm:exclusivity}, states that the two types of elicitability are mutually exclusive: a set-valued functional is either selectively elicitable or exhaustively elicitable, or not elicitable at all, subject to mild regularity conditions.
The proof follows from a refinement of the classical result that CxLS are necessary for elicitability (Proposition~\ref{prop:CxLS}).
This mutual exclusivity result is powerful in its ability to rule out elicitability of one type or the other; for example, quantiles of random variables are known to be selectively elicitable, thus failing to be exhaustively elicitable.
Interestingly, any specification of the quantile,
such as the lower quantile, or Value at Risk in the risk management literature, is also not elicitable in general; see Proposition~\ref{prop:selections b} and the discussion thereafter.

We illustrate our framework with new results for prediction intervals (Section~\ref{sec:PIs}) and Vorob'ev quantiles (Section~\ref{sec:Vorob'ev Quantiles}).
For prediction intervals, we show that 
the whole class of $\a$-prediction intervals is exhaustively elicitable. While this immediately rules out the selective elicitability, we consider certain interesting specifications of prediction intervals. If the midpoint or an endpoint is given by a constant, this specification is elicitable, indeed. However, if the midpoint or an endpoint is given via a general identifiable functional, it is in general not elicitable, unless the endpoint is specified via a quantile. We also show that the shortest prediction interval is not elicitable in either sense. This section complements and generalises recent results established independently in \cite{BrehmerGneiting2020}.
We then establish the exhaustive elicitability and selective identifiability of Vorob'ev quantiles of random closed sets.
For an application to systemic risk measures, we refer the reader to \cite{FisslerHlavinovaRudloff_RM}, where the theoretical framework of this paper has been applied to establish selective identifiability, exhaustive elicitability, and mixture representations of exhaustive scoring functions, leading to Murphy diagrams.
We close with a comprehensive literature review 
on forecast evaluation for set-valued quantities, covering spatial statistics, machine learning, engineering, climatology and meteorology, and philosophy, leaving many interesting avenues for future research. Technical proofs and further interesting results are deferred to the Appendix.

\section{Two types of elicitability and identifiability}\label{sec:theory}

\subsection{Scoring and identification functions for single-valued functionals}
\label{subsec:consistent scoring functions}
We use the decision-theoretic framework described for example in \cite{Gneiting2011}; cf.\ \cite{Savage1971}, \cite{Osband1985}, \cite{LambertETAL2008}, \cite{FisslerZiegel2016, FisslerZiegel2019}. 
Let $(\Omega, \mathfrak F, \bP)$ be some complete, atomless probability space rich enough to accommodate all random elements mentioned in the sequel.
With $Y$ we denote an observation of interest, taking values 
in some measurable space $(\O, \mathcal O)$, called \emph{observation domain}. 
Forecasts for $Y$ are denoted by $X$, taking values in a measurable space $(\A, \mathcal A)$ called \emph{action domain}. 
We assume that the directive for an ideal forecast is given in terms 
of a statistical functional of the (conditional) distribution $F$ of $Y$ (given $X$).
Mathematically, this is a map $T\colon\M\to\A$, where $\M$ is some class of probability measures or probability distribution functions on $(\O, \mathcal O)$.
All functions are tacitly assumed to be measurable. 
 A scoring function is a map $S\colon\A\times\O\to\R^* := (-\infty,\infty]$. It is negatively oriented, meaning that a forecast $x\in\A$ receives the penalty $S(x,y)$ if $y\in\O$ materialises. 
Statistically, the relative quality of a prediction observation sequence $(X_t, Y_t)$, $t=1, \ldots, N$, is evaluated by $S$ in terms of the \emph{realised score}
\be{eq:realised score}
\frac{1}{N}\sum_{t=1}^N S(X_t,Y_t).
\ee
Invoking an expected utility maximisation argument or a suitable law of large numbers, 
it has been widely argued \citep{MurphyDaan1985, EngelbergManskiETAL2009} that a scoring function should incentivise truthful forecasts in that the Bayes-act coincides with the given directive. 
We say that $S$ is incentive compatible or  \emph{$\M$-consistent} for $T$ if 
\be{eq:consistency 1}
\bar S(T(F),F)  \le \bar S(x,F)
\ee
for all $x\in \A$ and $F\in\M$ where we implicitly assume that $\bar S(x,F):=\int S(x,y) \diff F(y)$ exists. If additionally equality in \eqref{eq:consistency 1} implies that $x=T(F)$, $S$ is \emph{strictly }$\M$-consistent for $T$. 
A functional that admits a strictly consistent scoring function is called \emph{elicitable} \citep{Osband1985, LambertETAL2008}.
As such, the elicitability of a functional opens the way to meaningful forecast comparison \citep{Gneiting2011} which is closely related to comparative backtests in finance \citep{FisslerETAL2016, NoldeZiegel2017}. Similarly, it is crucial for $M$-estimation \citep{Huber1967, HuberRonchetti2009} and regression, such as quantile regression \citep{KoenkerBasset1978, Koenker2005} or expectile regression \citep{NeweyPowell1987}.

While scoring functions serve the purpose of forecast comparison and ranking, we employ \emph{identification functions} when it comes to forecast \emph{validation}.
Similarly to a scoring function, an identification function is a map $V\colon\A\times \O\to\R^k$ where we again make the tacit assumption that $\bar V(x,F) := \int V(x,y)\diff F(y)$ exists for all $x\in\A$, $F\in\M$ with the additional assumption that the expectation be finite.
$V$ is an \emph{$\M$-identification function} for $T$ if $\bar V(T(F),F)=0$. It is a \emph{strict} $\M$-identification function for $T$ if additionally $\bar V(x,F)=0$ implies that $x=T(F)$. 
In the literature on point-valued functionals, it turned out to be appropriate that $k$ coincides with the dimension of the forecasts \citep{Osband1985, FrongilloKash2014b, FisslerZiegel2016}.
Since statistical practice demands to evaluate the \emph{realised identification function}, which is the counterpart of \eqref{eq:realised score} upon replacing $S$ by $V$, $V$ simply needs to map to a real vector space.
One can even be more flexible and use an infinite-dimensional space. E.g.\ in Proposition \ref{prop:sel id} the identification function maps to $\R^{[-1,1]}$, the space of functions from $[-1,1]$ to $\R$. 
One can also relax the requirement that the expected identification function attains a 0 at the correctly specified forecast. It can rather attain some predefined particular value(s)---the important requirement being that this value be \emph{identifiable} in the common sense.

In statistics and econometrics, identification functions are often called \emph{moment functions} and give rise to the \emph{(generalised) method of moments} \citep{NeweyMcFadden1994} or $Z$-estimation. For a discussion of identifiability and calibration in the context of backtesting risk measures, we refer the reader to the insightful papers \cite{Davis2016} and \cite{NoldeZiegel2017}.
For a recent general perspective on identification, see \cite{BasseBojinov2020}.

\subsection{Selective and exhaustive scoring and identification functions}

When $T$ is set-valued, we may write \(
\T\colon \M \to 2^W,
\)
where $W$ is some generic space. 
As mentioned in the Introduction, we distinguish two types of forecasts with the corresponding notions of scoring and identification functions. In decision-theoretic terms, this translates into 
two sensible choices for the action domain $\A$:
\begin{enumerate}[(i)]
\item
$\A = \As \subseteq W$: The elements of the action domain $\As$ representing possible forecasts are \emph{points} in the space $W$. Truthful reporting means that there are generally \emph{multiple} best actions, namely all \emph{selections} $t\in\T(F)\subseteq \As$ for $F\in\M$. Mnemonically, we shall refer to $\As$ as a \emph{selective} action domain.
\item
$\A= \Ae \subseteq 2^W$: The elements of the action domain $\Ae$ representing possible forecasts are \emph{subsets} of the space $W$. Truthful reporting means that there is a \emph{unique} best action, namely the \emph{exhaustive} functional $\T(F)\in\Ae$ for $F\in\M$. Similarly, we shall refer to $\Ae$ as an \emph{exhaustive} action domain.
\end{enumerate}
The two different choices of action domains lay claim to different levels of precision and ambition of the forecasts. 
For a certain functional $T\colon\M\to 2^W$, the connection between the choice of the selective action domain $\As\subseteq W$ and the exhaustive action domain $\Ae \subseteq 2^W$ will be specified if needed for a certain result, otherwise remaining unspecified. However, a sensible connection between the two choices we have in mind is 
\(
\As = \bigcup_{B\in\Ae} B\,.
\)

We continue to use the dichotomy introduced above also for scoring functions evaluating forecasts for some set-valued functional $\T\colon \M \to 2^W$. Let $\M'\subseteq\M$ be some generic subset.

\begin{defn}[Consistency, elicitability]
\begin{enumerate}[\rm (i)]
\item

A \emph{selective} scoring function $S_{\text{\rm sel}}\colon\As\times \O\to\R^*$ is $\M'$-\emph{consistent} for $T\colon\M\to 2^{\As}$ if
\be{eq:weak cons}
\bar S_{\text{\rm sel}} (t, F) \le \bar S_{\text{\rm sel}} (x,F) \qquad \forall x\in\As, \ \forall t\in T(F), \ \forall F\in\M'.
\ee
The selective score $\Ss$ is \emph{strictly $\M'$-consistent} for $T$ if it is $\M'$-consistent for $T$ and if equality in \eqref{eq:weak cons} implies that $x\in T(F)$.
$T$ is \emph{selectively elicitable} on $\M'$ if there is a strictly $\M'$-consistent selective scoring function for $T$.

\item
An \emph{exhaustive} scoring function $S_{\text{\rm exh}}\colon\Ae\times \O\to\R^*$ is $\M'$-\emph{consistent} for $\T\colon\M\to \Ae$ if
\be{eq:strong cons}
\bar S_{\text{\rm exh}}(T(F), F) \le \bar S_{\text{\rm exh}}(B,F) \qquad \forall B\in\Ae, \ \forall F\in\M'.
\ee
The exhaustive score $\Se$ is \emph{strictly $\M'$-consistent} for $T$ if it is $\M'$-consistent for $T$ and if equality in \eqref{eq:strong cons} implies that $B=T(F)$.
$T$ is \emph{exhaustively elicitable} on $\M'$ if there is a strictly $\M'$-consistent exhaustive scoring function for $T$.
\end{enumerate}
Unless mentioned explicitly otherwise, we tacitly assume that all scoring functions are $\M'$-finite in the sense that $ \bar S_{\text{\rm sel}} (x,F), \bar S_{\text{\rm exh}}(B,F)<\infty$ for all $x\in \As$, $B\in \Ae$, $F\in\M'$. 
\end{defn}
If we merely say that $T\colon\M\to2^W$ is (selectively or exhaustively) elicitable, we mean it is (selectively or exhaustively) elicitable on $\M$.
Assuming $\M$-finiteness is convenient in many proofs and is a standard assumption in the literature \citep{Ziegel2016, FisslerZiegel2016, BrehmerStrokorb2019, WangWei2020}.
Note that without stipulating $\M$-finiteness, the \emph{strict} $\M$-consistency of a selective (exhaustive) scoring function $S_{\text{sel}}$ ($ S_{\text{exh}}$) implies that $\bar S_{\text{sel}}(t,F) \in\R$ for all $F\in\M$, $t\in T(F)$ ($\bar S_{\text{exh}}(T(F),F) \in\R$ for all $F\in\M$).
According to \cite{GneitingRaftery2007} we call any two (selective or exhaustive) scoring functions $S, S'\colon\A\times\O\to\R$ equivalent if there is some $\lambda>0$ and some function $a\colon\O\to\R$ such that $S'(x,y)=\lambda S(x,y)+a(y)$. It is immediate to see that this equivalence relation preserves $\M$-consistency, and also strict $\M$-consistency, subject to $a$ being $\M$-integrable.
If there is no risk of confusion, we shall drop the indices ``sel'' and ``exh'' to indicate the difference between selective and exhaustive interpretations, respectively.

For identification functions, we again make the distinction between selective and exhaustive identification functions to allow for a rigorous treatment of set-valued functionals. 
\begin{defn}[Identification function, identifiability]\label{defn:identification functions}
\begin{enumerate}[\rm (i)]
\item
A map $\Vs\colon\As\times \O\to\R^k$ is a \emph{selective $\M'$-identification function} for $\T\colon\M \to2^{\As}$ if $\bar V_{\text{\rm sel}}(t,F)=0$ for all $t\in\T(F)$ and for all $F\in\M'$. Moreover, $V_{\text{\rm sel}}$ is a strict selective $\M'$-identification function for $\T$ if 
\be{eq:V selective def}
\bar V_{\text{\rm sel}}(x,F) = 0 \quad \Longleftrightarrow \quad  x\in\T(F), \qquad \forall x\in\As, \quad \forall  F\in\M'.
\ee
$T$ is \emph{selectively identifiably} on $\M'$ if it possesses a strict selective $\M'$-identification function.
\item
A map $V_{\text{\rm exh}}\colon\Ae\times \O\to\R^k$ is an \emph{exhaustive $\M'$-identification function} for $\T\colon\M \to\Ae$ if $\bar V_{\text{\rm exh}}(\T(F),F)=0$ for all $F\in\M'$. Moreover, $V_{\text{\rm exh}}$ is a strict exhaustive $\M'$-identification function for $\T$ if 
\begin{equation}\label{eq:V exhaustive def}
\bar V_{\text{\rm exh}}(B,F) = 0 \quad \Longleftrightarrow \quad  B=\T(F), \qquad \forall B\in\Ae, \quad \forall F\in\M'.
\end{equation}
$T$ is \emph{exhaustively identifiably} on $\M'$ if it possesses a strict exhaustive $\M'$-identification function.
\end{enumerate}
\end{defn}

Again, we say that $T\colon\M\to2^W$ is (selectively or exhaustively) identifiable, if we mean it is (selectively or exhaustively) identifiable on $\M$.\\

For single-valued functionals such as the mean, the distinction between selective and exhaustive elicitability is obsolete, since any choice of an action domain leads to a unique best action. Hence, one is actually always in the exhaustive setting, and there is no point in mentioning this fact explicitly. Of course, we could formally identify a point-valued functional $T\colon\M\to\A$ with the set-valued functional $T'\colon\M\to \A' = \{\{a\}\,|\,a\in\A\}$ where $T'(F) = \{T(F)\}$. Then the following lemma holds.

\begin{lem}\label{lem:basic}
Let $T\colon\M\to\A$ be some point-valued functional. Define the set-valued functional $T'(F):= \{T(F)\}$, $F\in\M$. Then $T'$, considered as a map to the power set $2^\A$, is selectively elicitable (identifiable) if and only if $T'$, considered as a map to the exhaustive action domain $\A' = \{\{a\}\,|\,a\in\A\}$, is exhaustively elicitable (identifiable).
Moreover, the selective elicitability (identifiability) of $T'\colon \M\to 2^\A$ is equivalent to the elicitability (identifiability) of $T$.
\end{lem}

While we are aware of contributions to the literature which consider either the selective or the exhaustive interpretation only (see Section~\ref{subsec:literature}), one novelty in the present paper is that we thoroughly study and compare these two alternative notions, which is the content of the next section.

\section{Structural results}\label{sec:structural results}

The structural results presented in this section consist of generalisations of the classical Convex Level Sets (CxLS) property due to \cite{Osband1985} and their immediate implications (Section \ref{subsec:CxLS}), the main result on the mutual exclusivity of selective and exhaustive elicitability (Section \ref{subsec:mutual excl}), and implications of certain specifications of set-valued functionals in Section \ref{subsec:specifications}.

\subsection{CxLS properties and their implications}
\label{subsec:CxLS}

\begin{defn}\label{defn:CxLS}
Let $\T\colon \M\to 2^W$ be a set-valued functional and $\M'\subseteq\M$.
\begin{enumerate}[\rm (i)]
\item
$\T$ has the \emph{selective} CxLS property on $\M'$ if for all $F_0, F_1\in\M'$ and for all $\lambda\in(0,1)$ such that $(1-\lambda)F_0 + \lambda F_1\in\M'$:
\[
 \T(F_0) \cap \T(F_1) \subseteq \T\big((1-\lambda)F_0 + \lambda F_1\big).
\]
\item
$\T$ has the \emph{selective} CxLS* property on $\M'$ if for all $F_0, F_1\in\M'$ and for all $\lambda\in(0,1)$ such that $(1-\lambda)F_0 + \lambda F_1\in\M'$:
\[
\T(F_0) \cap \T(F_1)\neq \emptyset \ \implies \ \T(F_0) \cap \T(F_1) = \T\big((1-\lambda)F_0 + \lambda F_1\big).
\] 
\item
$\T$ has the \emph{exhaustive} CxLS property on $\M'$ if for all $F_0, F_1\in\M'$ and for all $\lambda\in(0,1)$ such that $(1-\lambda)F_0 + \lambda F_1\in\M'$:
\[
\T(F_0) = \T(F_1) \  \implies \  \T(F_0) = \T\big((1-\lambda)F_0 + \lambda F_1\big).
\] 
\end{enumerate}
\end{defn}
If we omit to mention the class $\M'$ explicitly, we mean that $T$ has the corresponding CxLS property on $\M$.
The exhaustive CxLS property is the most common one in the literature, and the one used for point-valued functionals \citep{SteinwartPasinETAL2014, BelliniBignozzi2015, DelbaenETAL2016, WangWei2020}. The selective CxLS property follows the one proposed in \cite{Gneiting2011}, while the selective CxLS* property is novel. However, it is noteworthy that the recent paper \cite{BrehmerStrokorb2019} introduced the notion of \emph{max-functionals}. Using our notation, a real-valued functional $\T\colon\M\to\R$ is called a \emph{max-functional} if for any $F_0, F_1\in\M$ and $\lambda\in(0,1)$
\[
\T\big((1-\lambda)F_0 + \lambda F_1\big) = \max\big(\T(F_0), \T(F_1)\big).
\]
It is immediate that a real-valued functional $\T\colon\M\to\R$ is a max-functional if and only if the set-valued functional $\T^+(F):= [\T(F), \infty)$ satisfies the selective CxLS* property.\\
The following implications are immediate. 

\begin{lem}\label{lem:implications}
Let $\T\colon \M\to 2^W$ be a set-valued functional and $\M'\subseteq\M$.
\begin{enumerate}[\rm (i)]
\item
If $T$ has the selective CxLS* property on $\M'$, then it also has the selective and the exhaustive CxLS property on $\M'$.
\item
If $T$ is singleton-valued on $\M'$, then the selective CxLS property on $\M'$, the exhaustive CxLS property on $\M'$ and the selective CxLS* property on $\M'$ are equivalent.
\end{enumerate}
\end{lem}
The second point of Lemma~\ref{lem:implications} underpins why the distinction of the CxLS properties is obsolete for the point-valued case.

It is classical knowledge originating from the seminal work of \cite{Osband1985} that the exhaustive CxLS property is necessary for exhaustive elicitability and exhaustive identifiability, and that the selective CxLS property is necessary for selective elicitability and selective identifiability. 
Under additional regularity assumptions and for real-valued functionals, \cite{SteinwartPasinETAL2014} established that the CxLS property is also sufficient for both elicitability and identifiability.
A novelty is the following necessity-result.

\begin{prop}\label{prop:CxLS}
If $\T\colon \M\to 2^{\As}$ is selectively elicitable on $\M'\subseteq\M$, then it satisfies the selective CxLS* property on $\M'$.
\end{prop}

\begin{proof}
  Let $F_0, F_1 \in\M'$, $\l\in(0,1)$ such that $F_\l = (1-\l)F_0 + \l F_1\in\M'$, and suppose there exists $t\in \T(F_0)\cap\T(F_1)$.
  Let $S\colon \As\times \O\to\R^*$ be a strictly $\M'$-consistent selective scoring function for $\T$. 
Moreover, let $x\in\As$. Note that for $i\in\{0,1\}$
\[
\bar S(x,F_i) - \bar S(t,F_i) 
\begin{cases}
=0, & \text{if } x\in \T(F_i)\\
>0, & \text{if } x\notin \T(F_i) 
\end{cases}
\]
due to the strict $\M$-consistency of $S$. This implies that
\begin{align}\label{eq:linearity}
\bar S(x,F_\lambda)  - \bar S(t,F_\lambda)
&= (1-\lambda)\big( \bar S(x,F_0)  - \bar S(t,F_0)\big) + \lambda\big(\bar S(x,F_1)  - \bar S(t,F_1)\big)\\[0.5em] \nonumber
&\begin{cases}
=0, & \text{if } x\in \T(F_0)\cap \T(F_1)\\
>0, & \text{if } x\notin \T(F_0)\cap \T(F_1).
\end{cases}
\end{align}
The identity in \eqref{eq:linearity} stems from the fact that the expected score $\bar S(\cdot, \cdot)$ behaves ``linearly'' in its second argument, which is the integration measure.
Again, invoking the strict $\M$-consistency of $S$, the assertion follows.
\end{proof}

For our next result, we need to introduce a property which essentially excludes any degenerate cases of set-valued functionals, e.g.\ being singleton-valued.
\begin{defn}
A set-valued functional $T\colon\M\to 2^W$ has the \emph{proper-subset property} if there are $F,G\in\M$ such that 
\[
\emptyset\neq T(G)\subsetneq T(F)
\]
and for all $\varepsilon\in(0,1)$ there exists a $\l_0 \in (0,\varepsilon)$ such that $(1-\lambda_0)F+\lambda_0 G\in\M$.
\end{defn}

\begin{thm}\label{thm:CxLS}
If $T\colon\M\to\Ae$ satisfies the proper-subset property and the selective CxLS* property, it is not exhaustively elicitable.
\end{thm}

\begin{proof}
Assume $S$ is a strictly $\M$-consistent exhaustive scoring function for $T$. Let $F,G\in\M$ be such that $\emptyset \neq T(G)\subsetneq T(F)$. Then
\[
\bar S(T(F),F) - \bar S(T(G),F)<0<\bar S(T(F),G) - \bar S(T(G),G).
\]
The proper-subset property implies that there is a sufficiently small $\lambda_0\in(0,1)$, such that $(1-\lambda_0)F+\lambda_0 G\in\M$ and, exploiting the selective CxLS* property yielding $T((1-\lambda_0)F+\lambda_0 G) = T(G)$, we end up with
\begin{multline*}
\bar S(T(F),(1-\lambda_0)F+\lambda_0 G) - \bar S(T(G),(1-\lambda_0)F+\lambda_0 G) \\
= (1-\lambda_0)\big(\bar S(T(F),F) - \bar S(T(G),F)\big) + \lambda_0 \big(\bar S(T(F),G) - \bar S(T(G),G)\big) <0,
\end{multline*}
which violates the strict $\M$-consistency.
Note that the last inequality only holds under the tacit assumption that $S$ is $\M$-finite.
\end{proof}

\begin{rem}
Remarkably, the combination of the selective CxLS* property and the proper-subset property implies that there are $F, G\in\M$ with $T(F) \neq T(G)$ such that for all $\lambda\in(0,1)$ it holds that $T((1-\lambda)F+\lambda G) \in\{T(F), T(G)\}$. That means, in our Theorem~\ref{thm:CxLS} we directly recover the condition of Theorem 3.3 in \cite{BrehmerStrokorb2019}. Even though their result is stated for real-valued functionals only, it immediately generalises to the set-valued case. Hence, the conclusions coincide in both instances implying that $T$ fails to be (exhaustively) elicitable.
\end{rem}

\subsection{Mutual exclusivity}
\label{subsec:mutual excl}

We now present our main result, which states that, for functionals satisfying the proper-subset property, selective and exhaustive elicitability are mutually exclusive.
The proof follows immediately from Proposition~\ref{prop:CxLS} and Theorem~\ref{thm:CxLS}.

\begin{thm}[Mutual exclusivity]\label{thm:exclusivity}
  Let $T\colon \M\to \Ae\subseteq 2^{\As}$ be a set-valued functional with the proper-subset property. Then $T$ cannot be both selectively elicitable and exhaustively elicitable.
\end{thm}

This result gives a broad insight into the structure of set-valued elicitability.
It basically establishes the following partition of set-valued functionals: 
\begin{enumerate}[(1)]
\item
The class of selectively elicitable functionals.
\item
The class of exhaustively elicitable functionals.
\item
The class of functionals which are not elicitable at all.
\end{enumerate}
The result also gives a powerful tool to rule out elicitability without the need for a direct argument, which in some cases may appear quite challenging a priori.
We give several example applications of Theorem~\ref{thm:exclusivity} below.

\begin{exmp}\label{examples}
\begin{enumerate}[\rm (i)]
\item
Any $\a$-quantile, $\a\in(0,1)$ is selectively elicitable. If the class $\M$ is reasonably large (e.g.\ it contains all measures with finite support), then the $\a$-quantile clearly satisfies the proper-subset property. Hence, it fails to be exhaustively elicitable on such a class.
\item
If $\M$ is the class of distributions on $\R$ with finite support, then the mode functional is selectively elicitable on $\M$ with the strictly $\M$-consistent selective scoring function $S(x,y) = \one\{x\neq y\}$ \citep{Heinrich2014, Gneiting2017}. Since the mode functional satisfies the proper-subset property on $\M$, it also fails to be exhaustively elicitable on $\M$.
\item
Any elicitable real-valued functional $T\colon\M\to\R$ induces trivial set-valued functionals $T^-(F):= (-\infty,T(F)]$ and $T^+(F) = [T(F),\infty)$. Clearly, the elicitability of $T$ is equivalent to the exhaustive elicitability of $T^-$ and $T^+$ considered as maps to $\A^- = \{(-\infty, x]\,|\,x\in\R\}$ and $\A^+ = \{[x,\infty)\,|\,x\in\R\}$, e.g.\ by invoking the revelation principle \citep{Osband1985, Gneiting2011, Fissler2017}. If $T$ is not constant on $\M$, then $T^-$ and $T^+$ also satisfy the proper-subset property, which means they violate the selective CxLS* property such that they are not selectively elicitable.\\
\emph{Vice versa,} if $T^+$ or $T^-$ satisfy the selective CxLS* property, then $T$ or $-T$ is a max-functional in the sense of \cite{BrehmerStrokorb2019} such that $T$ (and $-T$) is not elicitable unless it is constant, which recovers their Corollary 3.4.
\item
In \cite{FisslerHlavinovaRudloff_RM}, the exhaustive elicitability of the set-valued systemic risk measures defined by \cite{FeinsteinRudloffWeber2017} has been established. For a random vector $Y$ representing a financial system, a measure of systemic risk is defined as a collection of capital allocations $k\in\R^d$ such that $\rho(\Lambda(Y+k))\leq0$ where $\rho$ is a scalar risk measure and $\Lambda\colon\R^d\to\R$ a non-decreasing aggregation function. The cash-invariance property of these risk measures implies that they satisfy the proper-subset property. This means that they cannot be selectively elicitable.
\item 
In Section~\ref{sec:PIs}, we consider the class of $\a$-prediction intervals, i.e., of intervals a random variable will fall into with a probability of at least $\a$. We show that, on a suitable class of probability distributions $\M$, this class of $\a$-prediction intervals is exhaustively elicitable on $\M$, and in Lemma~\ref{lem:subset example} we construct an example that shows that the proper-subset property is satisfied on $\M$. The combination of these results then rules out the selective elicitability of the class of $\a$-prediction intervals on $\M$.
\item
Theorem \ref{prop:SI} shows that there are classes of distributions where the collection of all shortest $\a$-prediction intervals fails to be elicitable in either sense---selectively and exhaustively; see Remark \ref{rem:double failure} for details. 
Up to our knowledge, this is the first non-degenerate  example of a set-valued functional which is not elicitable in either sense. (Clearly, any non-elicitable single-valued functional, such as the variance, would trivially satisfy such a statement by virtue of Lemma \ref{lem:basic}).
\item
In Section~\ref{sec:Vorob'ev Quantiles}, we establish that Vorob'ev quantiles of random sets are selectively identifiable and exhaustively elicitable. Under the additional mild proper-subset property, which is satisfied in a lot of settings, this means that Vorob'ev quantiles cannot be selectively elicitable.
\end{enumerate}
\end{exmp}

\begin{rem}\label{rem:id exclusivity}
  Elicitability and identifiability have structural differences in this context.
  While Theorem~\ref{thm:CxLS} carries over to exhaustive identifiability with an easy adaption of the proof, 
it does not seem to be possible to establish an analogon of Proposition~\ref{prop:CxLS} for selective identifiability due to possible cancellation effects. One can merely establish that selective identifiability implies the selective CxLS property. 
Therefore, it remains open if selective and exhaustive identifiability are mutually exclusive in the sense of Theorem~\ref{thm:exclusivity}.
\end{rem}

\subsection{Specifications of set-valued functionals}
\label{subsec:specifications}

Let us now take a look at \emph{specifications} of set-valued functionals. For a set $W\neq\emptyset$ and a set-valued functional $T'\colon\M\to 2^W$, a point-valued functional $T\colon\M\to W$ is a specification of $T'$ if $T(F)\in T'(F)$ for all $F\in\M$. We start with a lemma, the proof of which is straightforward.

\begin{lem}\label{lem:selections}
Let $T'\colon\M\to 2^{\A}$ be a set-valued functional with $T'(F)\neq\emptyset$ for all $F\in\M$, and let $T\colon\M\to\A$ be a specification of $T'$. 
\begin{enumerate}[\rm (i)]
\item
If $S\colon\A\times \O\to\R^*$ is a (strictly) $\M$-consistent selective scoring function for $T'$, then it is an $\M$-consistent scoring function for $T$.
\item
If $V\colon\A\times \O\to\R$ is a (strict) selective $\M$-identification function for $T'$, then it is an $\M$-identification function for $T$.
\end{enumerate}
\end{lem}

Clearly, the scoring function $S$ (identification function $V$) appearing in Lemma~\ref{lem:selections} is only \emph{strictly} consistent (\emph{strict}) for $T$ if $T'$ is a singleton on $\M$, that is, $T'(F) = \{T(F)\}$ for all $F\in\M$.
This suggests the question as to whether the specification can be elicitable (identifiable) at all, which the following proposition is concerned with.

\begin{prop}\label{prop:selections}
Let $T'\colon\M\to 2^{\A}$ be selectively elicitable and $T\colon\M\to\A$ a specification of $T'$. Let $\M_1:=\{F\in\M\,|\,T'(F)=\{T(F)\}\}$ and suppose that $\M\setminus\M_1\neq\emptyset$. \\
Let $\S'_{\M}$ ($\S'_{\M_1}$) be the class of strictly $\M$-consistent ($\M_1$-consistent) selective scoring functions for $T'$.
If $\S'_{\M}=\S'_{\M_1}$, then $T\colon\M\to\A$ is not elicitable.
\end{prop}

\begin{proof}
Let $\S_{\M}$ ($\S_{\M_1}$) be the class of strictly $\M$-consistent ($\M_1$-consistent) scoring functions for $T$. If $\S'_{\M_1} = \S'_{\M}$, it holds that 
\[
\S_{\M}\subseteq \S_{\M_1} = \S'_{\M_1} = \S'_{\M}.
\]
However, any $S'\in \S'_{\M}$ fails to be strictly $\M$-consistent for $T$. Hence, $\S_{\M} = \emptyset$.
\end{proof}

A common problem when applying Proposition~\ref{prop:selections} for practical purposes is that most characterisation results concerning the class of strictly consistent scoring functions, if known, typically assume regularity conditions on the scoring functions such as continuity or differentiability; cf.\ Table 1 in \cite{Gneiting2011b} or Osband's Principle \citep{Osband1985, FisslerZiegel2016}. 
Interestingly, an argument similar to the one used in the proof of Theorem~\ref{thm:CxLS} leads to a result which rules out the elicitability of specifications under very weak conditions on the functional. In particular, it dispenses with regularity conditions on scoring functions.

In line with \cite{BelliniBignozzi2015} we call a functional $T$ from a convex class of distributions to some topological space $\A$ mixture-continuous if for any $F_0, F_1\in\F$ the map $[0,1] \ni \lambda \mapsto T((1-\lambda)F_0+\lambda F_1)\in\A$ is continuous.

\begin{prop}\label{prop:selections b}
Let $T'\colon\M\to 2^{\A}$, $T'\neq \emptyset$, satisfy the selective CxLS* property. Suppose there are distributions $F,G,H\in\M$ such that 
\be{eq:disjoint}
T'(F)\cap T'(G) = \{t_1\}, \quad T'(F)\cap T'(H) = \{t_2\}, \quad \text{with }\ t_1\neq t_2.
\ee
Then, any specification $T\colon\M\to \A$ of $T'$ is neither elicitable nor identifiable.
Moreover, if $\A$ is a space with a Fr\'echet topology, that is, if for any $a, b\in\A$ with $a\neq b$ there is an open set $U\subseteq \A$ such that $a\in U$ and $b\notin U$, then any specification $T$ fails to be mixture-continuous. 
\end{prop}

\begin{proof}
Let $T\colon\F\to\A$ be a specification of $T'$ and suppose $S\colon \A\times\O\to\R^*$ is a strictly $\M$-consistent scoring function for $T$. Let $F,G,H\in\M$ satisfy \eqref{eq:disjoint} with $t_1\neq t_2$ specified there. 
The selective CxLS* property implies that for any $\lambda\in(0,1)$ we have that $t_1 = T((1-\lambda)F+\lambda G)$ and $t_2 = T((1-\lambda)F + \lambda H)$. Then, for $t_0 = T(F)$ we have that $t_0\neq t_1$ or $t_0 \neq t_2$. Without loss of generality, assume $t_0\neq t_1$. 
The map $\gamma\colon [0,1]\to \A$, $\lambda \mapsto T((1-\lambda)F + \lambda G)$ is neither injective nor constant, such that Lemma B.1 in \cite{FisslerZiegel2019} implies that $T$ is not identifiable.\\
Assume that there is a strictly $\M$-consistent scoring function $S$ for $T$. This implies that for all $\lambda\in(0,1)$
\[
\bar S(t_0,F) - \bar S(t_1,F) < 0 < \bar S(t_0,(1-\lambda)F + \lambda G) - \bar S(t_1,(1-\lambda)F + \lambda G).
\]
This contradicts the elementary fact that the map $[0,1]\ni\lambda \mapsto \bar S(t_0,(1-\lambda)F + \lambda G) - \bar S(t_1,(1-\lambda)F + \lambda G)$ is continuous (where we have exploited the $\M$-finiteness of $S$), which rules out the elicitability of $T$.
Finally, if $\A$ has a Fr\'echet topology, $\gamma$ is not continuous, which shows that $T$ is not mixture-continuous. Indeed, let $U\subset\A$ be an open set such that $t_0\in U$, but $t_1\notin U$. Then $\gamma^{-1}(U) = \{0\}$, which is not open in $[0,1]$.
\end{proof}

We would like to emphasise that the mere failure of mixture-continuity of $T$ does not rule out its elicitability. Indeed, Proposition 2.2 in \cite{FisslerZiegel2019} (cf.\ Proposition 3.4 in \cite{BelliniBignozzi2015}) only rules out the existence of a \emph{continuous} strictly consistent scoring function for $T$.

\begin{prop}\label{prop:quantile selections}
Let $\M$ be a convex class of distributions such that \eqref{eq:disjoint} is satisfied for the $\a$-quantile, $\a\in(0,1)$. Then no specification of the $\a$-quantile is elicitable.
\end{prop}
\begin{proof}
Since the $\a$-quantile is selectively elicitable (see e.g.\ \cite{Gneiting2011b}), the claim is directly implied by Proposition~\ref{prop:selections b}.
\end{proof}
Note that \eqref{eq:disjoint} is satisfied for the $\alpha$-quantile e.g.\ if $\M$ contains all distributions with finite support. Proposition~\ref{prop:quantile selections} thus rules out the elicitability of the lower quantile or the specification introduced in the recent preprint \cite{AronowLee2018} relative to such classes.

Notably, Proposition~\ref{prop:selections b} rules out the elicitability with an $\M$-finite score, which also translates to Proposition~\ref{prop:quantile selections}. Relaxing this condition and looking at the 0- and 1-quantile, we arrive at functionals which are both selectively and exhaustively elicitable, which is the content of Subsection~\ref{subsec:essential}.

\section{Prediction intervals}\label{sec:PIs}

A common task for the statistical forecaster 
is to report an interval $[a,b]\subseteq  \R$ into which future observations of a given real-valued random variable $Y$ will fall with at least a specified coverage probability $\alpha\in(0,1]$, that is, $\bP(Y\in [a,b])\ge\alpha$. Thereby, the inherent uncertainty of the actual outcome is captured.
Any such interval will be referred to as an $\alpha$-prediction interval. 

The literature on evaluating prediction intervals considers reports for these functionals typically in the exhaustive sense, meaning that an interval is reported rather than a single point.\footnote{Reporting a single point in an $\alpha$-prediction interval does not bear much information such that selective forecasts are of no practical interest.
Concerning exhaustive forecasts, any interval can be identified with its endpoints and therefore the exhaustive elicitability (identifiability) is equivalent to the elicitability (identifiability) of a single vector, making use of the so called \emph{revelation principle}, originating from \cite[p.\ 9]{Osband1985}; see \cite[Theorem 4]{Gneiting2011}.} 
\citet[Sections 6.2 and 9.3]{GneitingRaftery2007} consider consistent exhaustive scores for the central $\a$-prediction interval or `equal-tailed' $\a$-prediction intervals; cf.\ \cite{Greenberg2018} for a discussion of these scores and \cite{BracherETAL2020} for a timely application of interval forecasts in the context of epidemiology. This basically amounts to a prediction for a pair of quantiles at the $(1-\a)/2$- and $(1-(1-\a)/2)$-level. If one fixes a certain coverage of, say, $\a$, this ansatz can be generalised to construct consistent scoring functions for a non-central $\alpha$-prediction interval of which the endpoints are specified in terms of quantiles at level $\beta$ and $\beta + \alpha$, where $\beta\in(0, 1-\alpha)$.
\cite{SchlagWeele2015} also consider exhaustive scoring functions for interval-valued predictions. However, they start with a certain scoring function of appeal to them and do not thoroughly characterise the functional which is elicited by this scoring function.
See \cite{AskanaziETAL2018} for an overview of interval forecasts, in which, however, mostly impossibility results are presented.

There is typically a whole class of $\alpha$-prediction intervals for $Y$, 
resulting in a collection of subsets of $\R$. In Section \ref{sec:ident_elicitPis} we show that this whole class of $\alpha$-prediction intervals is exhaustively elicitable, subject to sensible conditions on the class of distributions. As a direct consequence of Theorem~\ref{thm:exclusivity}, it is not selectively elicitable.
This fact imposes a substantial challenge to the sound evaluation of single arbitrary $\alpha$-prediction intervals without imposing any further restrictions.
On the other hand, imposing such further restrictions, it is well known that an $\alpha$-prediction interval given by two quantiles as its endpoints can be elicited due to the elicitability of the individual quantiles, if the quantiles are singletons on the respective class of probability distributions. Such an interval is a particular specification of the class of $\alpha$-prediction intervals, and one might wonder about the elicitability (identifiability) of other specifications.
In the second part of this section, we discuss the elicitability and identifiability of several specifications of the class of $\alpha$-prediction intervals, with largely negative results.

Our results are nicely complemented by the very recent and independently developed preprint \cite{BrehmerGneiting2020}. They essentially study the subclass of $\a$-prediction intervals with exact coverage $\a$, and show that this subclass fails to be selectively elicitable; see Remark \ref{rem:comparison} for details. Furthermore, they establish properties on homogeneous and translation invariant scores for the central $\a$-prediction interval (or `equal-tailed' $\a$-prediction interval) and show some complementary impossibility results on the shortest $\a$-prediction interval.

\subsection{Notation}

Let $\M_0$ be the class of Borel probability distributions on $\R$ where we deliberately overload notation and identify the corresponding Borel measures with their cumulative distribution functions. Let $\bar U := \{(a,b)^\intercal \in \bar \R^2\,|\,a\le b\}$, where $\bar\R\colon =\R\cup\left\{-\infty,\infty\right\}$, and $U := \{(a,b)^\intercal \in \R^2\,|\,a\le b\}$. 
For any $F\in\M_0$ and for any $(a,b)^\intercal \in \bar U$, we set $F([a,b]):= F([a,b]\cap \R) = F(b) - F(a-)$ where $F(b) := F((-\infty,b])$ if $b<\infty$ and $F(\infty) = F(\R) =1$, and $F(a-) := F((-\infty,a))$. 
For $\beta\in[0,1]$ and $F\in\M_0$, recall the definition of the $\beta$-quantile, $q_\beta(F)$, and the lower $\beta$-quantile, $q_{\beta}^-(F)$, of $F$ as
\begin{align*}
q_{\beta}(F) &:= \{t\in\bar \R\,|\,F((-\infty,t))\le \beta \le F((-\infty,t])\}\,,\\
q_{\beta}^-(F) &:= \inf q_{\beta}(F) = \inf\{ t\in\bar \R\,|\,\beta \le F((-\infty,t])\}\,.
\end{align*}
We introduce the following subclasses of $\M_0$: Let $\M_{\text{inc}}$ be the class of 
strictly increasing distribution functions, 
$\M_{\text{cont}}$ the class of continuous distributions, and $\M_{\text{inc,cont}}:= \M_{\text{inc}}\cap \M_{\text{cont}}$.
Naturally, the fact that $F\in\M_{\text{inc}}$ implies that the support of $F$ is whole $\R$. However, to allow for the treatment of distributions with bounded support, we define the class of \emph{$\a$-pseudo-increasing distributions} $\M_{\a,\text{inc}}$ for $\a\in(0,1]$. 
For any $F\in\M_0$ we say that $F\in\M_{\a,\text{inc}}$ if and only if 
\begin{gather}
\# q_{\a}(F) = \# q_{1-\a}(F)=1, \quad \text{and}\\ 
 \forall a\in \R \ \text{s.t.} \ F(a-)<1-\a: 
\#q_{\a+F(a-)}(F) >1  \implies  q_{F(a-)}(F) = \{a\},
\end{gather}
where the notation $\# A$ denotes the cardinality of a set $A$. This means that $F\in\M_{\a,\text{inc}}$ if and only if its $\a$- and $(1-\a)$-quantiles are singletons, and if for any $\beta\in(0,1-\alpha)$ the $\beta$-quantile or the $(\beta+\a)$-quantile is a singleton.

For any $\a\in(0,1]$ and upon identifying any non-empty interval $[a,b]\subseteq\R$ with the vector of its endpoints $(a,b)^\intercal \in \bar U$,
we formally introduce the class of $\a$-prediction intervals for a distribution $F\in\M_0$ as \be{eq:I_a}
\I_\a(F): = \{(a,b)^\intercal \in \bar U \,|\,F([a,b])\ge\a\}\,.
\ee
Clearly, $(a,b)^\intercal \in \I_\a(F)$ implies that $(a-x_1,b+x_2)^\intercal \in \I_\a(F)$ for all $x_1, x_2\ge0$. That is, $\I_\a(F)$ is an upper set with respect to the ordering cone $C := (-\infty, 0]\times[0,\infty)$. Moreover, $\I_\a(F)$ is non-empty since $F(\R) =1 \ge\a$, implying that $(-\infty, \infty)^\intercal\in\I_\a(F)$.
Therefore, we introduce the natural maximal exhaustive action domain for reports for $\I_\a$ as 
\[
\mathcal U = \{\emptyset\neq A \subseteq \bar U\,|\, A = A+C\}\,,
\]
with the usual definition of the Minkowski sum.
Moreover, for any $F\in\M_0$, we introduce the following functions closely connected to $\I_\a(F)$: First, the function
\begin{gather}
\nonumber
\Gamma_\a(F) \colon \{a\in\bar \R\,|\,F(a-)\le 1-\a\} \to (-\infty,\infty], \\ \label{eq:u}
a\mapsto \Gamma_\a(F)(a) = \inf\{b\ge a\,|\,F([a,b])\ge\a\} = q^-_{\a+F(a-)}(F)
\end{gather}
gives the upper endpoint of the shortest $\a$-prediction interval with lower endpoint $a$; see Figure~\ref{fig:Gamma} for an illustration. 
\begin{figure}
\centering
\includegraphics[width = 0.7\textwidth]{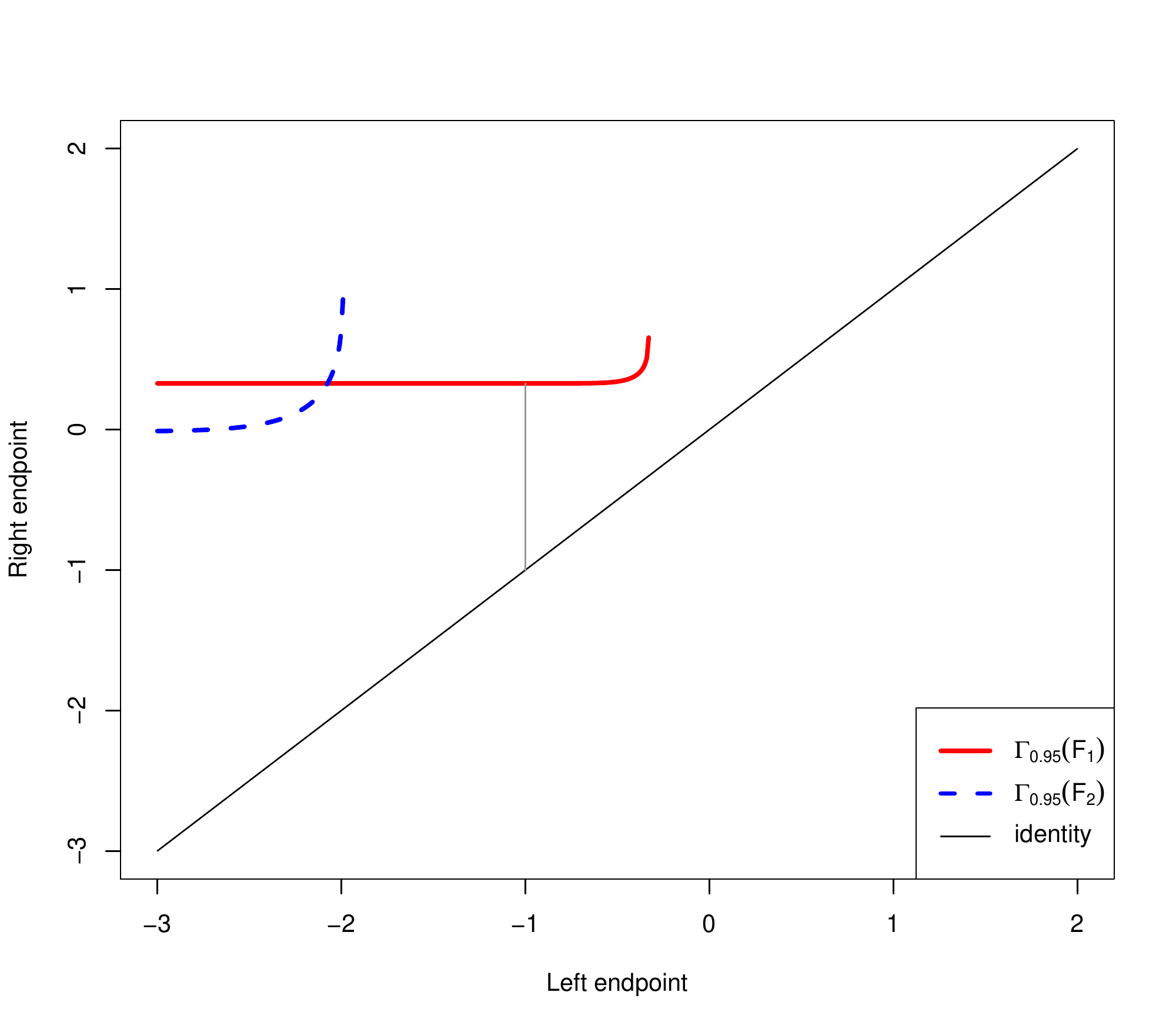}
\caption{Graphs of $\Gamma_{0.95}(F_1)$ (red solid), $F_1 = \mathcal N(0,0.2^2)$, and $\Gamma_{0.95}(F_2)$ (blue dashed), $F_2 = \mathcal N(-1,0.6^2)$. In light grey, the shortest $0.95$-prediction interval of $F_1$ with lower endpoint $-1$ is depicted.
}
\label{fig:Gamma}
\end{figure}
\begin{lem}\label{lem:Gamma}
\begin{enumerate}[\rm (i)]
\item
For all $F\in\M_0$ and for all $a\in\bar \R$ such that $F(a-)\le1-\a$ it holds that $F\big([a,\Gamma_\a(F)(a)]\big)\ge\a$.
\item
For all $F\in\M_0$ the function $\Gamma_\a(F)$ is increasing and left-continuous.
\item
For $F\in\M_{\text{\rm inc,\,cont}}$ the function $\Gamma_\a(F)$ is strictly increasing and continuous.
\item
For $F\in\M_{\a,\,\text{\rm inc}}$ it holds that $\lim_{a\to-\infty} \Gamma_{\a}(F)(a) = q^-_{\a}(F)$, and $\Gamma_\a(F)^{-1}(\{\infty\}) \in \{\emptyset, q_{1-\a}(F)\}$, which means that $\infty$ is attained at most once.
\end{enumerate}
\end{lem}
\begin{proof}
(i) follows from the right continuity of $F$. (ii) is implied by the fact that the functions $a\mapsto F(a-)$ and $\beta\mapsto q_{\beta}^-(F)$ are increasing and left-continuous. For (iii) recall that for $F\in\M_{\text{\rm inc,\,cont}}$, both $a\mapsto F(a-)$ and $\beta\mapsto q_{\beta}^-(F)$ are strictly increasing and continuous. Finally, for (iv) note that for $F\in\M_{\a,\,\text{\rm inc}}$, $\beta\mapsto q_{\beta}^-(F)$ is continuous at $\a$ and $1-\a$. Since $\lim_{a\to-\infty}F(a-) = 0$, the first assertion follows. For the latter, if there is no $a\in\R$ such that $F(a-)=1-\a$, the function $\Gamma_\a(F)$ is clearly finite. If there is some $a_0\in\R$ such that $F(a_0-) = 1-\a$ then $a_0\in q_{1-\a}(F) = \{q^-_{1-\a}(F)\}$ such that it is unique. $\Gamma_\a(F)(q^-_{1-\a}(F)) = q^-_1(F) = \infty$ if and only if the support of $F$ is unbounded from above. As $q_{1-\a}(F)$ is a singleton, for all $b<a_0$ it holds that $F(b-)<F(a_0-)$. Then $\a+F(b-)<1$ and therefore $\Gamma_\a(F)(b)$, being the lower $(\a+F(b-))$-quantile of $F$, is finite.
\end{proof}
\noindent
Similarly to $\Gamma_\a(F)$, we introduce for $F\in\M_0$
\begin{gather}
 \label{eq:d}
d_\a(F) \colon \R \to [0,\infty], \qquad
m\mapsto d_\a(F)(m) = \inf\{c\ge 0\,|\,F([m-c,m+c])\ge\a\}\,,
\end{gather}
which gives (half of) the length of the shortest $\a$-prediction interval of $F$, centred at $m$. Again, a continuity argument yields that the infimum is attained such that $F\big([m-d_\a(F)(m), m+d_\a(F)(m)]\big)\ge\a$. Moreover, for $\a<1$, $d_\a(F)$ takes finite values only.

Note that $\I_\a(F)$ corresponds to the epigraph of $\Gamma_\a(F)$, given by
\[
\epi \Gamma_\a(F) := \{(a,b)^\intercal \in \bar U \,|\,  b\ge \Gamma_\a(F)(a), \ F(a-)\le 1-\a\}, \quad F\in\M_0\,,
\]
which also guarantees the (Borel-) measurability of the set $\I_\a(F)$.
We also introduce the graph of $\Gamma_\a(F)$ as 
\[
\graph \Gamma_\a(F) := \big\{\big(a,\Gamma_\a(F)(a)\big)^\intercal \in  \bar U \,|\, F(a-)\le 1-\a\big\}, \quad F\in\M_0\,.
\]
Finally, we introduce the subclass $\mathcal U^*\subseteq\mathcal U$ of sets which can be written in form of epigraphs of left-continuous functions $\gamma\colon [-\infty, b] \to (-\infty, \infty]$, for some $b\in\R$ such that $\gamma^{-1}(\{\infty\})\in\{\emptyset, \{b\}\}$ and such that $\lim_{a\to-\infty}\gamma(a) = \gamma(-\infty)$.

\subsection{Elicitability and identifiability of the class of $\alpha$-prediction intervals}\label{sec:ident_elicitPis}

One of the main results of this paper is as follows.

\begin{thm}\label{thm:intervals_class}
For $\alpha\in(0,1]$ the following assertions hold:
\begin{enumerate}[\rm (i)]
\item The functional $F\mapsto \graph \Gamma_\alpha(F)$ is selectively identifiable on $\M_{\text{\rm inc,\,cont}}$ with the strict selective $\M_{\text{\rm inc,\,cont}}$-identification function 
\[
\Vs\colon U\times\R \to\R ,\quad (x,y)\mapsto\Vs(x,y)=\one\{y\in [x_1,x_2]\}-\alpha.
\]
Moreover in $\M_0$, $\Vs$ is still a selective $\M_0$-identification function for $\graph \Gamma_\alpha(\cdot)$ and it is 
oriented in the sense that $\bar V(x,F)\ge0$ if and only if $x\in\I_\a(F)$ for any $F\in\M_0$.

\item Let $\mu$ be a finite, $\sigma$-additive, nonnegative measure on $U$. The function $\Se\colon\mathcal U\times\R\to\R$
\begin{align}\label{eq:exh_score}
\Se(A,y) &= - \int_{A\cap U} \Vs(x,y)\dint\mu(x) 
= \a\mu(A) - \mu\big(((-\infty,y]\times[y,\infty))\cap A\big)
\end{align}
is an $\M_0$-consistent exhaustive scoring function for $\I_\alpha$. 
\item
If additionally $\mu$ is positive on $U$,\footnote{That means any nonempty open subset of $U$ has positive measure under $\mu$.}  then the restriction of $\Se$ to $\mathcal U^*\times \R$ is strictly $\M_{\a,\text{\rm inc}}$-consistent for $\I_\alpha$, rendering the class of $\alpha$-prediction intervals exhaustively elicitable on $\M_{\a,\text{\rm inc}}$. 
\end{enumerate}
\end{thm}

\begin{proof}
Part (i) follows from Lemma~\ref{lem:Gamma} and standard arguments.
For part (ii) let $F\in\M_0$, $A^* = \I_\a(F)$ and $A\in\mathcal U$. Then, using a Fubini argument, we obtain
\begin{align}\label{eq:proof}
&\Se(A,F)-\Se(A^*,F)=\int\limits_{(A^* \setminus A)\cap U}\bar{V}_{\text{sel}}(x,F)\dint\mu(x)-\int\limits_{(A \setminus A^*)\cap U}\bar{V}_{\text{sel}}(x,F)\dint\mu(x)\ge0.
\end{align}
The inequality is easily established by recalling that $\bar{V}_{\text{sel}}(x,F)\ge0$ if and only if $x\in A^*$. 
The proof of part (iii) is deferred to 
Appendix \ref{subsec:proofs}.
\end{proof}

Note that in part (iii), the fact that $q_{F(x_1-)}(F)$ is a singleton whenever $\#q_{\a+F(x_1-)}(F)>1$ plays an important role. If this were not the case, we would obtain rectangles of points $x\in\I_\a(F)$ with $\bar V_{\text{sel}} (x,F)=0$ with positive measure under $\mu$---namely $q_{F(x_1-)}(F)\times q_{\a+F(x_1-)}(F)$. This would mean that our exhaustive scoring function fails to distinguish between the correct forecast and one that does not contain some of the points within this rectangle. The reasoning behind why the $\a$ and $(1-\a)$ quantiles are required to be singletons is similar.\\
For $\a=1$, note that $\M_{\a,\text{\rm inc}}$ only contains distributions with support $\R$. But in that case, $\I_\a$ is constant, namely $\{\R\}$ for all distributions, and thus not interesting. We will therefore exclude the case $\a=1$ from further discussion.

\begin{rem}
Imposing the normalisation condition $S(A,y)\ge0$ with equality if and only if $A= \I_\a(\delta_y) = [-\infty,y]\times[y,\infty]$, it is straightforward to construct a score equivalent to the one in \eqref{eq:exh_score} given by 
\be{eq:S normalised}
\begin{split}
 S_\mu(A,y) &= \Se(A,y) - \Se([-\infty,y]\times[y,\infty],y) \\
&= (1-\a)\mu\big((-\infty,y]\times[y,\infty)\big) +\a\mu(A) - \mu\big(((-\infty,y]\times[y,\infty))\cap A\big)\\
&= (1-\a)\mu\big(((-\infty,y]\times[y,\infty))\setminus A\big) +\a\mu(A\setminus ((-\infty,y]\times[y,\infty))\big) .
\end{split}
\ee
From this stage, one can easily construct a family of elementary scores, $S_u = S_{\delta_u}$, $u\in U$, given by \eqref{eq:S normalised}. 
As a consequence of Theorem~\ref{thm:intervals_class}, these elementary scores are $\M_{\a,\,\text{inc}}$-consistent for $\I_\a$. Clearly, $S_\mu(A,y)=\int S_u(A,y)\mu(\diff u)$ which is a mixture representation in the spirit of \cite{EhmETAL2016}. 
This opens the way to the powerful tool of Murphy diagrams $u \mapsto S_{u}(A,y)$ discussed there as well. In order to avoid the necessity of choosing a measure $\mu$, one instead considers the elementary scores in \eqref{eq:S normalised} over different values of the parameter $u\in U$. In the one-dimensional case discussed in \cite{EhmETAL2016} as well as in the case of the class of $\a$-prediction intervals, one can easily visualise the values of the expected score differences graphically. With the possibly increasing dimensionality of the space $u$ comes from, the illustrative accessibility of this approach gets more involved. We discuss an example with possibly higher dimension in Section~\ref{sec:Vorob'ev Quantiles}.
For an illustration of 2-dimensional Murphy diagrams, we refer the reader to \cite{FisslerHlavinovaRudloff_RM}. 
\end{rem}

Intuitively, the class of $\a$-prediction intervals, $\I_\a(F)$, of a distribution $F$ contains a great deal of information about $F$ itself. So one might wonder if it is possible to recover $F$, knowing $\I_\a(F)$. If so, this would mean that $\I_\a$ actually constitutes a bijection. And consequently, the exhaustive elicitability of $\I_\a$ would directly follow from the existence of strictly proper scoring rules for probabilisitic forecasts \citep{GneitingRaftery2007}, invoking the \emph{revelation principle} \citep{Osband1985, Gneiting2011}. 
The following proposition asserts that $\I_\a$ is not a bijection, which underlines the novelty of Theorem~\ref{thm:intervals_class}.
\begin{prop}
  \label{prop:Ialpha-not-injective}
For $\a\in(0,1)$ the functional $\mathcal{I}_\alpha$ is not injective on $\M_{\a, \text{\rm inc}}$.
\end{prop}

\begin{proof}
  For $a,b\in\R$, $a,b>0$, define $a \% b := a - b \lfloor a/b \rfloor$, the real analog of the modulus.

  If $1 \% \alpha = 0$, then $\alpha = 1/n$ for some integer $n$, and for any $a \in [0,1]$ define the distributions $F_a \in \M_{\a, \text{\rm inc}}$ with density $f_a(y) = 1 - a \cos(2\pi ny)$ for $y\in[0,1]$ and $f_a(y) = 0$ otherwise.
  For all $a\in[0,1]$ and $x\in[0,1-\alpha]$, observe that $\int_x^{x+\alpha} f_a(y) \diff y = \alpha - a\int_x^{x+1/n} \cos(2\pi ny) \diff y = \alpha$.
  Hence, we have $\I_\a(F_a) = \{(x,x+\alpha)^\intercal : x\in[0,1-\alpha]\} \cup \{(z,\alpha)^\intercal:z\leq 0\} \cup \{(1-\alpha,z)^\intercal:z\geq 1\}$ for all $a\in[0,1]$, violating injectivity.

  Otherwise, let $\beta = 1 \% \alpha > 0$.
  For all $0 \leq a \leq \beta/(\alpha-\beta)$, define the probability density
  \begin{equation*}
    \label{eq:1}
    f_a(y) =
    \begin{cases}
      0 & y \notin [0,1]
      \\
      1-a(\alpha/\beta-1) & y \% \alpha \leq \beta
      \\
      1+a & y \% \alpha > \beta.
    \end{cases}
  \end{equation*}
  Thus, $f_0$ is the uniform density on $[0,1]$, and for $a>0$, $f_a$ raises and lowers the density according to where $y$ falls modulo $\alpha$.
  Letting $F_a \in \M_{\a, \text{\rm inc}}$ be the corresponding probability measure, we will show that $\I_\a(F_a) = \I_\a(F_0)$ for all $a > 0$.

  We again see that $\I_\a(F_0) = \{(x,x+\alpha)^\intercal : x\in[0,1-\alpha]\} \cup \{(z,\alpha)^\intercal:z\leq 0\} \cup \{(1-\alpha,z)^\intercal:z\geq 1\}$.
  For $F_a$, note that the Lebesgue measure of the set $\{y\in[x,x+\alpha]:y\%\alpha \leq \beta\}$ is exactly $\beta$ for all $x$.
  Thus, when $x \in [0,1-\alpha]$, we have $F_a([x,x+\alpha])
  = \beta(1-a(\alpha/\beta-1)) + (\alpha-\beta)(1+a)
  = \beta - a(\alpha-\beta) + (\alpha-\beta) + a(\alpha-\beta)
  = \alpha$, as desired.
  The cases $[z,\alpha]$ and $[1-\alpha,z]$ follow immediately.
\end{proof}

\begin{rem}
Variants of prediction intervals other than connected intervals might also be natural to consider, e.g., wrapped intervals (allowing intervals of the form $(-\infty,b] \cup [a,\infty)$ where $b<a$), unions of intervals, and most generally, any measurable prediction set.
  In Appendix \ref{sec:inject-results-pred}, we show that most of these generalisations are indeed bijective with $F$. That means their exhaustive elicitability follows directly from the existence of strictly proper scoring rules for probabilisitic reports and the revelation principle. 
  One exception is the case of wrapped intervals when $\alpha$ is rational, as the construction in the first case of Proposition~\ref{prop:Ialpha-not-injective} applies, and injectivity fails.
  (When $\alpha$ is irrational, repeatedly wrapping intervals corresponds to an irrational rotation, from which one can compute a dense set of quantiles such that one can again invoke the revelation principle to obtain exhaustive elicitability.)
  We claim that the class of wrapped prediction intervals with a rational $\alpha$ is exhaustively elicitable under mild assumptions on the underlying class of distributions, using a similar integral construction as the one in Theorem~\ref{thm:intervals_class}.
\end{rem}

In order to use Theorems~\ref{thm:exclusivity} and \ref{thm:intervals_class} to conclude that $\I_\a$ is not selectively elicitable on $\M_{\a, \text{inc}}$ it is essential to show that $\I_\a$ satisfies the proper-subset property on $\M_{\a, \text{inc}}$.
\begin{lem}\label{lem:subset example}
$\I_\a$ satisfies the proper-subset property on $\M_{\a, \text{\rm inc}}$ for $\a\in(0,1)$.
\end{lem}
\begin{proof}
We first show the claim for $\a\in[1/2,1)$. We determine $\I_\a(F)$ for $F=\text{Unif}([b,c])$, a uniform distribution on $[b,c]$, $b<c$. One easily verifies that $\Gamma_\a(\text{Unif}([b,c]))(a) = \max(a,b) + \a(c-b)$ for $a\le c-\a(c-b)$. For larger $a$, $\Gamma_\a(\text{Unif}([b,c]))$ is not defined. A straight forward calculation shows that for any $b<0$ and any $1-b(1-\a)/\a\le c\le 1- b\a/(1-\a)$, the domain of $\Gamma_\a(\text{Unif}([b,c]))$ is contained in the domain of $\Gamma_\a(\text{Unif}([0,1]))$ and that $\Gamma_\a(\text{Unif}([b,c]))\ge \Gamma_\a(\text{Unif}([0,1]))$ where the two functions do not coincide. As a result
$\emptyset\neq \I_\a(\text{Unif}([b,c]))\subsetneq \I_\a(\text{Unif}([0,1]))$; see the left panel of Figure~\ref{fig:subset}. 
Moreover, any convex mixture of two uniform distributions is an element of $\M_{\a, \text{inc}}$.\\
For $\a\in(0,1/2)$, it holds that $\I_\a(\delta_0) = [-\infty,0]\times[0,\infty]$ while $\I_\a(\delta_0/2 + \delta_1/2) = \big([-\infty,0]\times[0,\infty]\big) \cup \big([-\infty,1]\times[1,\infty]\big)$; see the right panel of Figure~\ref{fig:subset}. Note that even though there are $\l\in(0,1)$ such that $(1-\l)\delta_0 + \l (\delta_0/2 + \delta_1/2)\notin \M_{\a, \text{inc}}$, the proper-subset property is still satisfied.
\end{proof}

\begin{figure}
\centering
\includegraphics[width = 0.47\textwidth]{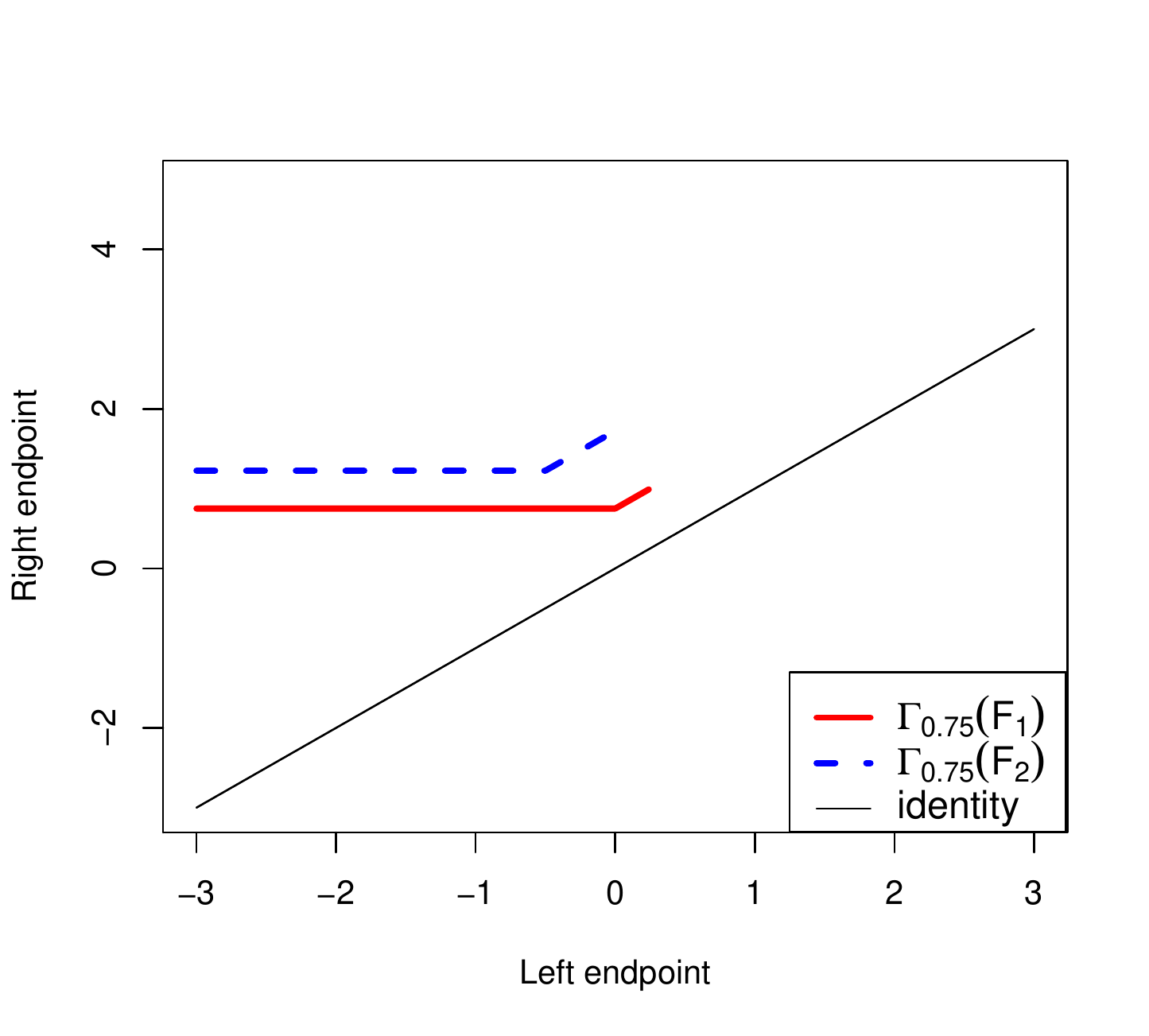}
\includegraphics[width = 0.47\textwidth]{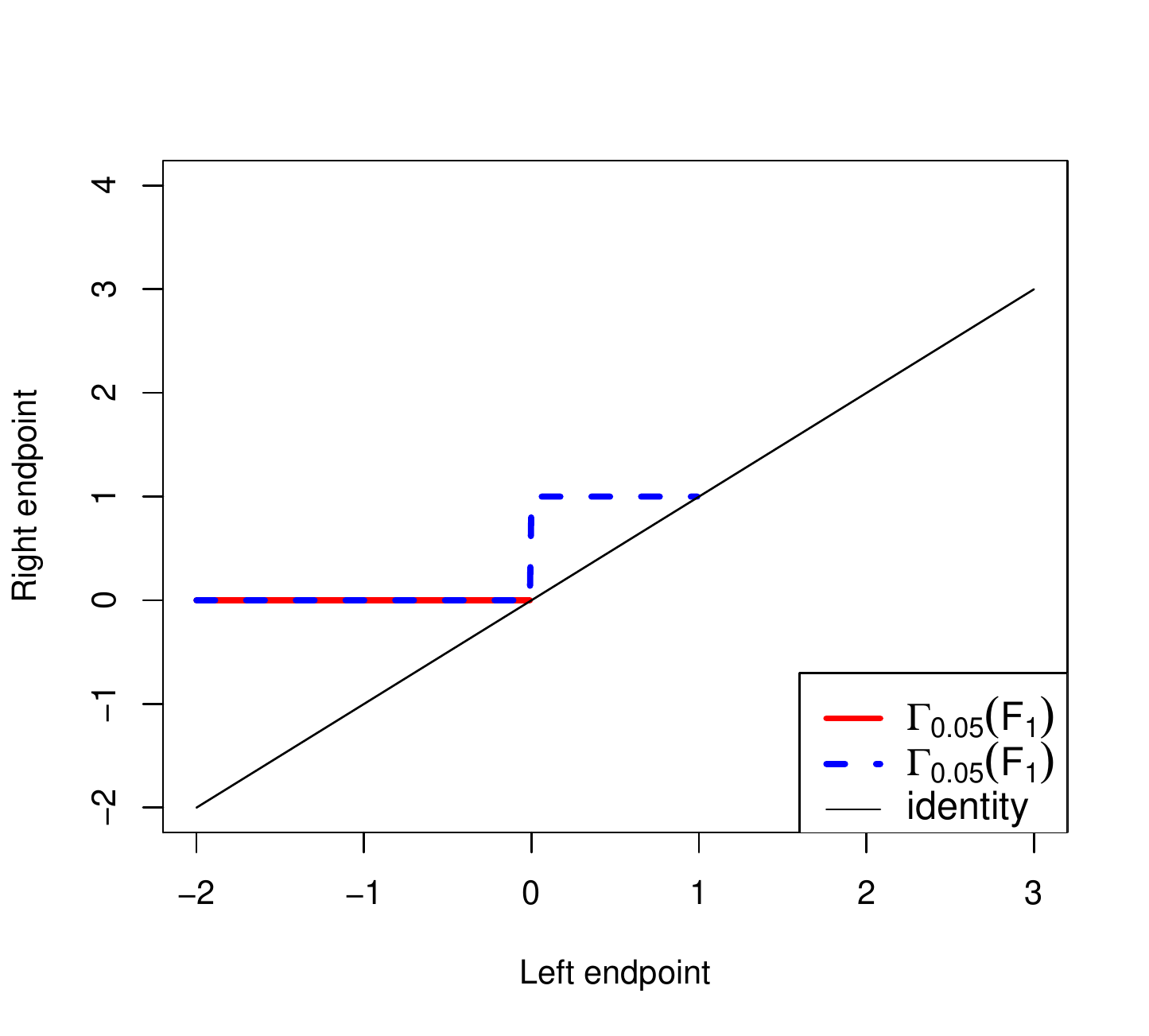}
\caption{Illustration of the proper-subset property.
Left panel: Graphs of $\Gamma_{0.75}(F_1)$ (red solid), $F_1= \text{Unif}([0,1])$, and $\Gamma_{0.75}(F_2)$ (blue dashed), $F_2 = \text{Unif}([-0.5,1.8])$. 
Right panel: Graphs of $\Gamma_{0.05}(F_1)$ (red solid), $F_1 = \delta_0$, and $\Gamma_{0.05}(F_2)$ (blue dashed), $F_2=\delta_0/2 + \delta_1/2$.}
\label{fig:subset}
\end{figure}

\begin{cor}\label{cor:not sel el}
For $\a\in(0,1)$, the class $\I_\alpha$ of $\alpha$-prediction intervals is not selectively elicitable on $\M_{\a, \text{\rm inc}}$.
\end{cor} 
\begin{proof}
This is a direct combination of Theorem~\ref{thm:intervals_class}, Theorem~\ref{thm:exclusivity} and Lemma~\ref{lem:subset example}.
\end{proof}

\begin{rem}\label{rem:comparison}
  Corollary \ref{cor:not sel el} is related to the impossibility result established in \citet[Section 3.1]{BrehmerGneiting2020}, that the `guaranteed coverage interval at level $\a$' ($GCI_\a$) is not selectively elicitable; see also \citet[Prop.~7.6]{LambertShoham2009} for a related result. Roughly speaking, $GCI_\a$ coincides with the topological boundary of $\I_{1-\a}$, or $\graph \Gamma_{1-\a}$. Even though there is an obvious bijection between $GCI_\a$ and $\I_{1-\a} = \epi \Gamma_{1-\a}$, one cannot invoke Osband's revelation principle here. It would only hold for exhaustive elicitability and not for selective elicitability. Therefore, even though the results are closely related, they are complementary.
\end{rem}

\subsection{Prediction interval with an endpoint or the midpoint given by an identifiable functional}
\label{subsec:midpoint}

In some situations, one might be interested in prediction intervals with one endpoint or the midpoint specified as some (identifiable) functional.
The simplest situation arises if the midpoint or an endpoint is simply a constant; we defer the discussion to Appendix \ref{subsec:fixed points}.
Apart from constants, the most natural such functionals appear to be the mean or the median for the midpoint, while also other quantiles or expectiles might be interesting.
If one endpoint is specified in terms of some quantile, the other endpoint must be a quantile itself and the elicitability of the vector is obvious and well known (if the quantiles are both singletons); see e.g.\ \cite{GneitingRaftery2007} or Proposition~\ref{prop:QQ}, which recalls this result for the sake of completeness.
On the other hand, we can show that
there are no twice continuously differentiable exhaustive scoring functions (see Propositions~\ref{prop:l} and~\ref{prop:m}) for other functionals under mild conditions. In the case of the midpoint given by an identifiable functional, this even holds for the quantile. This gives rise to the conjecture that such intervals are in general not elicitable.
Despite their failure of being (smoothly) elicitable, these functionals are still identifiable, therefore possessing the CxLS property. This leads to the novel observation that, in the multivariate setting, the equivalence of the CxLS property with identifiability and elicitability established for one-dimensional functionals in \cite{SteinwartPasinETAL2014} fails to hold.
We only address the case of the left endpoint given by an identifiable functional and remark that the right endpoint case works \emph{mutatis mutandis}.

\begin{prop}\label{prop:QQ}
Let $QI_{\a, \beta}\colon\M\to U$ be a prediction interval given by two lower quantiles, i.e.\ $QI_{\a, \beta}(F)= \big(q_{\beta}^-(F),q_{\alpha+\beta}^-(F)\big)^\intercal \in U$ with $\beta\in(0,1-\a)$. 
The following assertions hold:
\begin{enumerate}[\rm (i)]
\item $QI_{\a, \beta}$ is identifiable on any subclass $\M$ of $\M_0$ such that the $\beta$ and $(\a+\beta)$-quantiles are singletons for all distributions in $\M$. The function 
\[
V\colon U\times\R\to\R^2,\qquad (x,y)\mapsto V(x,y) = \big(\one\{y\le x_1\} - \beta, \one\{y\le x_2\} -\a- \beta\big)^\intercal\,
\]
is a strict identification function on $\M$.
\item $QI_{\a, \beta}$ is elicitable on any subclass $\M$ of $\M_0$ such that the $\beta$ and $\a+\beta$-quantiles are singletons for all distributions in $\M$. Any sum of two strictly $\M$-consistent scoring functions for the respective quantiles is a strictly $\M$-consistent scoring function for $QI_{\a, \beta}$.
\end{enumerate}
\end{prop}
Note that, in fact, essentially any strictly consistent scoring function for $QI_{\a,\beta}$ is a sum of two strictly consistent scoring functions for the respective quantiles; see \citet[Proposition 4.2]{FisslerZiegel2016}. Very recently, \citet[Theorem 3.1]{BrehmerGneiting2020} characterised all translation invariant or positively homogeneous consistent scores for the central $\a$-prediction interval. 

Choosing $\beta=0$, reporting $QI_{\a,\beta}$ would boil down to reporting the lower $\a$-quantile, such that the identifiability and elicitability hold if (and only if) the $\alpha$-quantile is a singleton.
For the case $\beta=1-\a$, the second component of $QI_{\a,\beta}$ is the essential supremum. Therefore, we only obtain an identifiability result if all distributions in $\M$ are unbounded from above and the $(1-\a)$-quantiles are singletons. For elicitability results, we refer to Subsection~\ref{subsec:essential}.\\
Finally, we would like to remark that Proposition~\ref{prop:QQ} together with the mutual exclusivity result of Theorem~\ref{thm:exclusivity} implies that there cannot be a scoring function $\R\times\R\to\R$ such that the expected score is minimised on an interval between two quantiles, subject to very mild conditions on the class of distributions $\M$ (such that the proper-subset property is satisfied for $QI_{\a,\beta}$).

In Proposition \ref{prop:QQ}, we ensured the existence of the $\a$-prediction interval by restricting the range of $\beta$. 
Similarly, one has to restrict the class of probability distributions suitably to ensure the existence of an interval with the demanded coverage when the left endpoint is given by some general identifiable functional $l\colon\M\to\R$ where $\M$ is some subclass of $\M_0$. To assure that there is enough mass above $l(F)$, we write $\M_l=\{F\in\M\,|\,F(l(F)-)\leq1-\alpha\}$. For a midpoint specification in terms of an identifiable functional $m\colon\M\to\R$, such a restriction is not necessary.

\begin{prop}\label{prop:l}
Let $\a\in(0,1)$.
Let $l\colon \M\to\R$ be an identifiable functional with a strict $\M$-identification function $V_l\colon\R\times\R\to\R$.  
Set $b\colon \M_l\to [0,\infty)$, defined as $b(F) = \big(\Gamma_\a(F)(l(F)) - l(F)\big)/2$, which is half of the length of the shortest $\a$-prediction interval with lower endpoint $l$, and set
 $T_l\colon= (l,b)^\intercal \colon\M_l\to \R\times[0,\infty)$. Then the following assertions hold:
\begin{enumerate}[\rm (i)]
\item $T_l$ is identifiable on $\M_l\cap \M_{\text{\rm inc,\,cont}}$ with a strict identification function 
\[
V\colon \R\times[0,\infty) \times\R\to\R^2, \qquad V(z_1,z_2,y)=\big(V_l(z_1,y),\one\{y\in[z_1,z_1+ z_2]\}-\a\big)^\intercal.
\]
\item Assume that $\M_l$ is such that 
	\begin{enumerate}[\rm (a)]
	\item \label{ass a}
	$\M_l\cap \M_{\text{\rm inc,\,cont}}$ is convex;
	\item \label{ass b}
	for any $z\in \R\times (0,\infty)$ there are $F_1, F_2, F_3\in\M_l\cap \M_{\text{\rm inc,\,cont}}$ such that $0$ is in the interior of the convex hull of the set $\big\{\bar V(z,F_i)\,|\,i\in\{1,2,3\}\big\}\subset \R^2$;
	\item \label{ass c}
	$\bar V(\cdot,F)$ is continuously differentiable on $\R\times (0,\infty)$ for all $F\in\M_l\cap \M_{\text{\rm inc,\,cont}}$;
	\item \label{ass d}
	for any $(l^*,b^*)^\intercal\in\R\times (0,\infty)$ there are distribution functions $F_i\in \M_l\cap \M_{\text{\rm inc,\,cont}}$ with densities $f_i$, $i\in\{1,2,3,4\}$, such that
\begin{align*}
(l(F_1),b(F_1))^\intercal=(l(F_2),b(F_2))^\intercal&=(l(F_3),b(F_3))^\intercal=(l(F_4),b(F_4))^\intercal=(l^*,b^*)^\intercal\\
\bar{V}_l'(l^*,F_1)&=\bar{V}_l'(l^*,F_2)=\bar{V}_l'(l^*,F_3)\\
f_1(l^*+2b^*)&=f_2(l^*+2b^*)\\
f_1(l^*+2b^*)&\neq f_3(l^*+2b^*)\\
f_1(l^*)&\neq f_2(l^*)\\
\bar{V}_l'(l^*,F_4)&\neq0.
\end{align*}
	\end{enumerate}
	Then there is no strictly $\M_l\cap \M_{\text{\rm inc,\,cont}}$-consistent scoring function $S$ for $T_l$ such that $\bar S(\cdot, F)$ is twice continuously differentiable on $\R\times (0,\infty)$ for any $F\in \M_l\cap \M_{\text{\rm inc,\,cont}}$. 
\end{enumerate}
\end{prop}
\begin{proof}
See Appendix~\ref{subsec:proofs}.
\end{proof}

Points (a), (b) and (d) are basically richness assumptions on the class $\M_l\cap \M_{\text{\rm inc,\,cont}}$, which are needed to establish necessary conditions on the shape of possible strictly consistent scoring functions via Osband's principle \cite[Theorem 3.2]{FisslerZiegel2016}. In particular, (b) and (d) in combination with the convexity stipulated under (a) are surjectivity condition where (b) also assumes that the expected identification function may vary enough. (c) is a pure smoothness assumption which is needed since the proof exploits first and second order conditions.
In concrete situations, e.g.\ when $\M_l\cap \M_{\text{\rm inc,\,cont}}$ is the class of finite Gaussian mixtures and $l$ is the mean functional, these conditions can be verified by straightforward calculations.

\begin{rem}
If $l$ is a lower $\beta$-quantile with $\beta<1-\alpha$, one can choose  $V_l(x,y)=\one\{y\leq x\}-\beta$. In this case $\bar{V}_l(x,F)=F(x)-\beta$ and $\bar{V}_l'(x,F)=f(x)$, thus there cannot be two distribution functions $F_1, F_2\in\M_l\cap \M_{\text{\rm inc,\,cont}}$ with $\bar{V}_l'(l^*,F_1)=\bar{V}_l'(l^*,F_2)$ and $f_1(l^*)\neq f_2(l^*)$. Hence, the elicitability of an interval given by two lower quantiles does not contradict Proposition~\ref{prop:l}.
\end{rem}

\begin{prop}\label{prop:m}
Let $\a\in(0,1)$.
Let $m\colon \M\to\R$ be an identifiable functional with a strict $\M$-identification function $V_m\colon\R\times\R\to\R$. 
Set $b \colon\M\to[0,\infty)$, defined as $b(F) = d_\a(F)(m(F))$, which is half of the length of the shortest $\a$-prediction interval with midpoint $m(F)$,
and set $T_m=(m,b)^\intercal\colon\M\to\A\colon=\R\times[0,\infty)$. Then the following assertions hold:
\begin{enumerate}[\rm (i)]
\item $T_m$ is identifiable on $\M\cap \M_{\text{\rm inc,\,cont}}$ with a strict identification function 
\[
V\colon\R\times [0,\infty)\times\R\to\R^2,\qquad  V(z_1,z_2,y)=\big(V_m(z_1,y),\one\{y\in[z_1- z_2, z_1+z_2]\}-\a\big)^\intercal.
\]
\item 
Assume that $\M$ is such that assumptions (a), (b) and (c) from Proposition~\ref{prop:l} hold mutatis mutandis. Moreover, suppose that (d) for any $(m^*,b^*)^\intercal\in\R\times (0,\infty)$, there are distribution functions $F_i,\in \M\cap \M_{\text{\rm inc,\,cont}}$ with densities $f_i$, $i\in\{1,2,3,4\}$, such that
\begin{align*}
(m(F_1),b(F_1))^\intercal=(m(F_2),b(F_2))^\intercal&=(m(F_3),b(F_3))^\intercal=(m(F_4),b(F_4))^\intercal=(m^*,b^*)^\intercal\\
\bar{V}_m'(m^*,F_1)&=\bar{V}_m'(m^*,F_2)=\bar{V}_m'(m^*,F_3)\\
f_1(m^*+b^*)+f_1(m^*-b^*)&=f_2(m^*+b^*)+f_2(m^*-b^*)\\
f_1(m^*+b^*)+f_1(m^*-b^*)&\neq f_3(m^*+b^*)+f_3(m^*-b^*)\\
f_1(m^*+b^*)-f_1(m^*-b^*)&\neq f_2(m^*+b^*)-f_2(m^*-b^*)\\
\bar{V}_m'(m^*,F_4)&\neq0.
\end{align*}
Then there is no strictly $\M\cap \M_{\text{\rm inc,\,cont}}$-consistent scoring function $S$ for $T_m$ such that $\bar S(\cdot, F)$ is twice continuously differentiable on $\R\times(0,\infty)$ for any $F\in\M\cap \M_{\text{\rm inc,\,cont}}$.
\end{enumerate}
\end{prop}
\begin{proof}
See Appendix~\ref{subsec:proofs}.
\end{proof}

\subsection{Shortest prediction intervals}
\label{subsec:essential}

In the context of probabilistic forecasts,
\citet[p.\ 243]{GneitingETAL2007} proposed the paradigm of ``maximizing the sharpness of the predictive distribution subject to calibration'', continuing: ``Calibration refers to the statistical consistency between the distributional forecasts and the observations and is a joint property of the predictions and the events that materialize. Sharpness refers to the concentration of the predictive distributions and is a property of the forecasts only.'' Following this rationale, a particularly well-motivated restriction of $\I_\a$ is the shortest prediction interval $SI_\a$, meaning a prediction interval of minimal length (sharp) subject to achieving a coverage of at least $\alpha$ (calibrated).
This is in line with the decision-theoretic derivation of the `prescriptive optimal interval forecast' given in \citet[Section 2.2]{AskanaziETAL2018}:
``restrict attention to correctly-calibrated intervals, and then pick the shortest (on average).''
In this subsection, we will study the elicitability of $SI_\a$.

Let us first consider the case $\a=1$, where the shortest $\a$-prediction interval of $F\in\M$ is  $SI_1(F) = ((\essinf(F), \esssup(F))^\intercal$, which is possibly of infinite length.
Here $\essinf$ and $\esssup$ are the essential infimum and supremum, respectively, defined by
$\sup q_0$ and $\inf q_1$, where $q_\alpha$ is the quantile functional.
Thus, to understand the elicitability of $SI_1$, it suffices to study the elicitability of $\essinf$ and $\esssup$.

To this end, let $g\colon\R\to\R$ be an increasing and bounded function, and set $g(\pm\infty) = \lim_{x\to\pm\infty}g(x)$.
Recall that for $\a\in(0,1)$ a consistent selective score for the $\a$-quantile is given by $S_\a(x,y) = (\one\{y\le x\}-\a) \big(g(x) - g(y)\big)$.
If $q_\alpha$ is surjective on $\M$ in the sense that for any $x\in\R$ there exists an $F\in\M$ such that $x\in q_\alpha(F)$, then $S_\a$ becomes strictly $\M$-consistent if and only if $g$ is strictly increasing.
Now consider the following generalisations of $S_\a$ for $\a\in\{0,1\}$, clearly failing to be $\M$-finite in general:
\begin{align}\label{eq:S_0}
S_0(x,y) &= \infty\cdot\one\{y<x\} + g(y) - g(x),\\
S_1(x,y) &= \infty\cdot\one\{y>x\} + g(x) - g(y).\label{eq:S_1}
\end{align}
Interestingly, if $g$ is constant, $S_0$ becomes a strictly $\M_0$-consistent selective scoring function for $q_0$, and $S_1$ for $q_1$.
On the other hand, if $g$ is strictly increasing, they become strictly $\M_0$-consistent for 
the essential infimum and essential supremum, respectively, and the elicitability of $SI_1$ then follows.

\begin{prop}\label{sec:SI-1}
  $SI_1$ can be elicited on $\M_0$ with non $\M_0$-finite, strictly $\M_0$-consistent score $S((a,b)^\intercal,y) = \infty\cdot\one\{y \notin [a,b]\} + g(b) - g(a)$ where $g\colon\R\to\R$ is strictly increasing and bounded. 
\end{prop}
\begin{proof}
  We have $\bar S((a,b)^\intercal,F) = \infty$ if $F([a,b]) < 1$, and $g(b) - g(a)$ otherwise.
Clearly, this is the sum of the strictly consistent functions for the essential supremum and infimum given in \eqref{eq:S_0} and \eqref{eq:S_1}.
\end{proof}

\begin{rem}
  Since $q_0 = (-\infty,\sup q_0]$ and $q_1 = [\inf q_1,\infty)$, the scores $S_0$ and $S_1$ can be directly used to construct strictly consistent \emph{exhaustive} scoring functions for $q_0$ and $q_1$, respectively, by invoking the revelation principle.
Thus, the 0-quantile and 1-quantile are both selectively and exhaustively elicitable.
Theorem~\ref{thm:exclusivity} does not apply here as Theorem~\ref{thm:CxLS} only holds for scoring functions whose expectation is always finite.
Moreover, if we were to impose that all scores be finite in expectation, a common assumption in the literature \citep{FisslerZiegel2016, BrehmerStrokorb2019, WangWei2020}, Proposition~\ref{prop:selections b} would apply, implying the non-elicitability of $\essinf$ and $\esssup$, recovering a result established in the proof of \citet[Corollary 4.3]{Ziegel2016}.
Hence, while the result for $SI_1$ is positive, it is narrowly so, as it leans heavily on the ability to assign infinite expected scores.
\end{rem}

Turning now to the case $\a\in(0,1)$, we first observe that the shortest $\a$-prediction interval is necessarily bounded: by a simple continuity argument for the probability measure $F\in\M$, there is some $C>0$ (depending on $F$) such that $(-C,C)^\intercal \in \I_\a(F)$.
For $F\in\M$ and $\a\in(0,1)$, we may therefore define
\be{eq:SI}
SI_\a(F) = \big\{(a,b)^\intercal \in \I_\a(F)\,|\, b-a \le d-c \ \text{ for all } (c,d)^\intercal \in \I_\a(F),\ c,d\in\R\big\}\,.
\ee
For many distributions $F$, such as the uniform distribution on $[0,1]$, $SI_\a(F)$ contains more than a single element, so we again must formally distinguish between exhaustive and selective reports.
Importantly, Lemma~\ref{lem:SI existence}, proven in the Appendix, asserts that $SI_\a(F)$ is always non-empty, meaning each distribution has at least one shortest $\a$-prediction interval.
\begin{lem}\label{lem:SI existence}
For all $\a\in(0,1]$ and for all $F\in\M$ it holds that $SI_\a(F)\neq \emptyset$.
\end{lem}

The following theorem gives a comprehensive negative result of the elicitability of $SI_\alpha$ for $\alpha\in(0,1)$.
\begin{thm}[Shortest $\a$-prediction interval]\label{prop:SI}
\begin{enumerate}[\rm (i)]
\item
  For $\a\in(0,1)$, the shortest $\a$-prediction interval $SI_\a$ is not selectively elicitable on any class $\M$ containing (a) all distributions with bounded Lebesgue densities, or (b) all distributions on $\N_0$ which are unimodal with mode $k$ for some $k\ge1$.   
  Moreover, $SI_\a$ can even not be selectively elicited with any non $\M$-finite score.
\item
For $\a\in(0,1]$, the shortest $\a$-prediction interval $SI_\a$ is not exhaustively elicitable (with $\M$-finite scores) on any class $\M$ such that for some $x\neq y\in\R$, $\{(1-\l)\delta_x + \l \delta_y\,|\,\l\in[0,1]\}\subseteq \M$.
\end{enumerate}
\end{thm}

\begin{proof}\parskip=0pt
  (i) If (b) holds, the assertion is an immediate consequence of \citet[Theorem 3.5]{BrehmerGneiting2020}.

  Suppose that (a) holds.
  Inspired by the argument given in \citet[Section 4.2]{frongillo2012general}, let $\alpha\in(0,1)$ and $G_1,G_2,G_3\in\M$ be the uniform distributions on the intervals $[0,1]$, $[1,2]$, and $[3,1+2/\alpha]$, respectively.
  Let $F_0 = \alpha G_1 + (1-\alpha) G_3$ and $F_1 = \tfrac 1 2 \alpha G_1 + \tfrac 1 2 \alpha G_2 + (1-\alpha) G_3$, and define $x_0 = (0,1)^\intercal$, $x_1 = (0,2)^\intercal$, so that we have $\{x_0\} = SI_\alpha(F_0)$ and $\{x_1\} = SI_\alpha(F_1)$.
  Now define $F_\lambda =  (1-\lambda) F_0 +\lambda F_1 $; by construction, $\{x_1\} = SI_\alpha(F_\lambda)$ for all $0 < \lambda \leq 1$.
  
  Now suppose for a contradiction that some scoring function $S$ was strictly $\M$-consistent for $SI_\alpha$. Here, we also dispense with the assumption that $S$ is $\M$-finite.
  We first argue $\bar S(x_i,G_j) < \infty$ for all $i\in\{0,1\}$ and $j\in\{1,2,3\}$.
  Let $F'=\alpha G_1 + \epsilon G_2 + (1-\alpha-\epsilon) G_3$ for $\epsilon = \tfrac 1 2 \min(\alpha,1-\alpha)$.
  The case $i=0$ follows as $x_0 \in SI_\alpha(F')$, so $\bar S(x_0,F') = \alpha \bar S(x_0,G_1) + \epsilon \bar S(x_0,G_2) + (1-\alpha-\epsilon) \bar S(x_0,G_3) < \infty$, and thus each term must be finite.
  For $i=1$, observing $x_1 \in SI_\alpha(F_1)$, the same reasoning gives the result.
  We conclude that $\bar S(x_i,F_j) < \infty$ for $i,j\in\{0,1\}$, as all constituent terms are finite.

  Define the function $\gamma:[0,1]\to\R$ by $\gamma(\lambda) = \bar S(x_0,F_\lambda) - \bar S(x_1,F_\lambda)$.
  Expanding by linearity of expectation, we have $\gamma(\lambda) =  (1-\lambda) (\bar S(x_0,F_0) - \bar S(x_1,F_0)) +\lambda (\bar S(x_0,F_1) - \bar S(x_1,F_1))$.
  We conclude that $\gamma$ is a continuous function as all terms above are finite.
  By strict $\M$-consistency, we now have $\gamma(\lambda) > 0$ for $0 < \lambda \leq 1$ and $\gamma(0) < 0$, contradicting the continuity of $\gamma$.
  (Cf. \citet[Corollary 4.12]{frongillo2012general}.)
  
 (ii) Without loss of generality, we assume that the mixtures $F_\l = (1-\l)\delta_0 + \l\delta_1$ are in $\M$ for all $\l\in[0,1]$. Using again the convention to identify any interval $[a,b]$ with the vector of its two endpoints $(a,b)^\intercal$, we obtain for $\beta\in(0,1/2]$ and $\gamma \in(1/2,1]$
\begin{align*}
SI_\beta(F_\l) = \begin{cases}
\{(0,0)^\intercal\}, & \l\in[0,\beta) \\
\{(0,0)^\intercal,(1,1)^\intercal\}, & \l\in [\beta,1-\beta] \\
\{(1,1)^\intercal\}, & \l\in(1-\beta,1],
\end{cases}
&&SI_\gamma(F_\l) = \begin{cases}
\{(0,0)^\intercal\}, & \l\in[0,1-\gamma] \\
\{(0,1)^\intercal\}, & \l\in (1-\gamma,\gamma) \\
\{(1,1)^\intercal\}, & \l\in[\gamma,1].
\end{cases}
\end{align*}
Clearly, for $\a\in(0,1]$, $SI_\a$ is non-constant on $\{F_\l\,|\,\l\in[0,1]\}$ and attains only finitely many values. Therefore,
as a direct consequence of \citet[Corollary 3.5]{BrehmerStrokorb2019}, $SI_\a$ there is no strictly $\M$-consistent and $\M$-finite exhaustive score for $SI_\a$. 
\end{proof}

While Theorem~\ref{prop:SI} (i) considers the question of selective elicitability of $SI_\a$ and is in line with the findings of \cite{BrehmerGneiting2020}, there is no counterpart to part (ii) dedicated to the exhaustive elicitability of $SI_\a$.

\begin{rem}
The classes $\M$ specified in (i) and (ii) of Theorem~\ref{prop:SI} are not contained in $\M_{\a,\text{\rm inc}}$, which makes it hard to thoroughly compare the results of Theorem~\ref{prop:SI} with Theorem~\ref{thm:intervals_class}. The elicitability of $SI_\a$ on $\M_{\a,\text{\rm inc}}$ remains an open problem, though we conjecture a negative result.\\
This should also be compared with the discussion of \citet[Condition 3.7 and Theorem 3.8]{BrehmerGneiting2020}. The class considered in their Theorem 3.8 is also not contained in $\M_{\a,\text{\rm inc}}$.
In particular, they also leave open the problem of elicitability on classes of distributions with strictly positive Lebesgue densities.
\end{rem}

\begin{rem}\label{rem:double failure}
On any class $\M$ satisfying the conditions of Theorem \ref{prop:SI} (i) \emph{and} (ii) $SI_\a$ fails to be selectively elicitable and exhaustively elicitable. This yields an interesting set-valued functional which fails to be elicitable in either sense.
\end{rem}

\begin{rem}\label{rem:region}
One may also consider general prediction \emph{regions} rather than merely intervals, in which case a natural object to study is the $\a$-prediction region with smallest Lebesgue-measure.
One can employ a very similar argument to the one used in the proof of Theorem~\ref{prop:SI} (ii) to rule out the exhaustive elicitability (with an $\M$-finite score) of the class of $\a$-prediction regions with minimal Lebesgue measure, denoted by $SR_\a$. Indeed, again writing $F_\l = (1-\l)\delta_0 + \l\delta_1$, we obtain 
for $\beta\in(0,1/2]$ and $\gamma \in(1/2,1]$
\begin{align*}
SI_\beta(F_\l) = \begin{cases}
\{\{0\}\}, & \l\in[0,\beta) \\
\{\{0\},\{1\}\}, & \l\in [\beta,1-\beta] \\
\{\{1\}\}, & \l\in(1-\beta,1],
\end{cases}
&&SI_\gamma(F_\l) = \begin{cases}
\{\{0\}\}, & \l\in[0,1-\gamma] \\
\{\{0,1\}\}, & \l\in (1-\gamma,\gamma) \\
\{\{1\}\}, & \l\in[\gamma,1].
\end{cases}
\end{align*}
The rest of the argument follows the lines of the proof.
\end{rem}

\begin{prop}\label{prop:sel id}
  The shortest $\a$-prediction interval is selectively identifiable on $\M_{\text{\rm inc,\,cont}}$ in the following sense.
  Let $[-1,1]^{\R}$ denote the space of all functions $\R \to [-1,1]$, and define the function-valued identification function $V\colon U\times \R\to [-1,1]^{\R}$ by
\[
 \R  \ni a \mapsto V\big((x_1,x_2)^\intercal ,y\big)(a) = \one\{y\in[x_1+a,x_2+a]\} -\a\,,
\]
for $(x_1,x_2)^\intercal \in U$ and $y\in \R$.
Then for any $F\in \M_{\text{\rm inc,\,cont}}$ and any $(x_1,x_2)^\intercal \in U$ it holds that $(x_1,x_2)^\intercal \in SI_\a(F)$ if and only if
\be{eq:cond id}
\bar V\big((x_1,x_2)^\intercal ,F\big)(a)\le0 \quad \forall a\in\R \quad \text{and }  \quad \bar V\big((x_1,x_2)^\intercal ,F\big)(0)=0\,.
\ee
\end{prop}
\begin{proof}
  Let $F\in \M_{\text{\rm inc,\,cont}}$ and $(x_1,x_2)^\intercal \in U$ such that \eqref{eq:cond id} holds. $\bar V\big((x_1,x_2)^\intercal ,F\big)(0)=0$ implies that $(x_1,x_2)^\intercal \in \I_\a(F)$ and $x_2 = \Gamma_\a(F)(x_1)$ such that there is no shorter $\a$-prediction interval for $F$ with a lower endpoint of $x_1$.
  The condition $\bar V\big((x_1,x_2)^\intercal ,F\big)(a)\le0$ for all $a\in\R$
  means that all other intervals $[x_1+a, x_2+a]$ with the same length either fail to be an $\a$-prediction interval (corresponding to a strict inequality), or they are also an $\a$-prediction interval (corresponding to equality),
  but with the same logic as above, there cannot be a shorter one with the same lower endpoint $x_1+a$. Hence, we can conclude that $(x_1,x_2)^\intercal \in SI_\a(F)$. Vice versa, if $(x_1,x_2)^\intercal \in SI_\a(F)$, then \eqref{eq:cond id} is immediate.
\end{proof}

In closing, we would like to remark that for multivariate observations, a generalisation from prediction intervals to prediction regions is mandatory.
If we do not impose any restrictions other than measurability, say, we can still obtain a selective identifiability result in the spirit of Theorem~\ref{thm:intervals_class}~(i).
For other similar extensions of our results, Remark \ref{rem:region} points into a negative direction for the smallest prediction regions. 
For considerations analogue to the ones in Subsection \ref{subsec:midpoint} one would need to impose further restrictions on the shape of the regions (e.g.\ one might consider balls with a certain centre and radius) to ask sensible questions. 
This is beyond the scope of the current project. On the other hand, the following section elaborates on a complementary direction of Vorob'ev quantiles, which only become interesting in a multivariate / spatial setting.

\section{Vorob'ev quantiles}
\label{sec:Vorob'ev Quantiles}

	As \cite{AzzimontiETAL2018} point out, the ``problem of estimating the set of inputs that leads a system to a particular behavior is common in many applications'', and they explicitly mention the fields of reliability engineering and climatology (see references therein). In such a context, the quantity of interest is a \emph{random set} $\bY$. This set could specify the region of a blackout in a country, the area affected by an avalanche in the mountains or tumorous tissue in the human body. In many situations such as extreme weather events, e.g.\ floods, storms or heatwaves, the random set $\bY$ is specified in terms of an excursion set $\{z\in\R^d\,|\,\xi_z\ge t\}$, $t\in\R$, of some random field $(\xi_z)_{z\in\R^d}$. 
A main task in mathematical statistics is to construct confidence intervals or confidence regions in $\R^d$ from a random sample. Consequently, such confidence regions may also be considered as random sets.
Functionals of interest are various expectations of $\bY$ as described in the comprehensive textbook \cite{Molchanov2017}, notably, the Vorob'ev expectation \citep{ChevalierETAL2013}, the distance average expectation \citep{AzzimontiETAL2016} and conservative estimates based on Vorob'ev quantiles \citep{AzzimontiETAL2018}. 

In this section, we shall focus on Vorob'ev quantiles and shall notably establish exhaustive elicitability results and related selective identifiability results under reasonable conditions. In that respect, it generalises and extends the known result that the symmetric difference in measure is an exhaustive consistent scoring function for the median; see Proposition 2.2.8 in \cite{Molchanov2017} and below for details.\\

To settle some notation, we work again on some suitable complete, atomless probability space $(\Omega, \mathfrak F, \bP)$. Let $E$ be some generic separable Banach space equipped with its Borel $\sigma$-algebra, with the Euclidean space as a leading example. Let $\mu$ be some $\sigma$-finite non-negative reference measure on $E$ and let $\mathfrak U$ be the family of closed subsets of $E$. We shall use the convention to denote any subset of $E$ with a capital latin letter, with the additional distinction that a random set will be denoted with a bold capital letter.
\begin{defn}[Random closed set]\label{defn:random set}
$\bY\colon\Omega\to\mathfrak U$ is called a \emph{random closed set} if for all compact sets $K\subseteq E$ 
\[
\left\{\omega\,|\, \bY(\omega)\cap K\neq\emptyset\right\}\in\mathfrak F.
\]
\end{defn}
In decision-theoretic terminology, that means that our observation domain $\O$ coincides with $\mathfrak U$. In line with Definition~\ref{defn:random set} and following  \citet[Chapter 1]{Molchanov2017}, we equip $\mathfrak U$ with the $\sigma$-algebra generated by the family $\mathfrak B(\mathfrak U):=\{ U \in \mathfrak U \colon U\cap K \neq \emptyset, \ K\ \in\mathfrak K\}$ where $\mathfrak K$ is the collection of all compact subsets of $E$. Consequently, we shall identify the distribution $F_{\bY}$ of a random closed set $\bY$ with its \emph{capacity functional}. That is, we set $F_{\bY}\colon \mathfrak K \to [0,1]$, $F_{\bY}(K) = \bP(\bY \cap K\neq \emptyset)$. As before, let $\M$ denote some generic class of distributions $\mathfrak K \to [0,1]$. \\
While $F_{\bY}$ characterises the whole (joint) distribution of a random closed set $\bY$, its restriction to singletons, in some sense, specifies the marginal distributions of $\bY$. This restriction is called \emph{coverage function} $p_\bY\colon E\to\left[0,1\right]$ and is formally defined as 
\[
p_\bY(u) := F_{\bY}(\{u\}) 
=\bP(u\in\mathbf{Y}).
\]
Finally, we can define the Vorob'ev quantiles of closed random sets.
\begin{defn}[Vorob'ev quantile]
The upper excursion set of $p_\mathbf{Y}$ at level $\alpha\in\left[0,1\right]$, 
\[
Q_\alpha(\mathbf{Y})\coloneqq\left\{u\in E\,|\, p_\bY(u)\geq\alpha\right\},
\]
is called the \emph{Vorob'ev $\alpha$-quantile} of $\bY$.
\end{defn}
The Vorob'ev $\alpha$-quantile plays a special role in the context of confidence regions. Suppose $\bY = g_\a(Y_1, \ldots, Y_n)$ where $Y_1, \ldots, Y_n$ are i.i.d.\ random vectors in $\R^m$ following some parametric distribution $F(\theta)$, $\theta\in\Theta\subseteq\R^k$, and $g_\a\colon(\R^m)^n \to \mathfrak U$ is a measurable map. In this context, $E$ clearly corresponds to $\Theta$. Then one can say that $\bY$ constitutes an $\a$-confidence region for the parameter $\theta$ if $\theta \in Q_\alpha(\bY)$ for all $\theta\in\Theta$. 

As $p_\bY$ is an upper semicontinuous function  \citep{Molchanov2017}, $Q_\alpha(\bY)$ is a closed set in $E$. Therefore, in a decision-theoretic terminology, we set the exhaustive action domain to be $\mathfrak U$ and the selective action domain to be $E$.
Moreover, for further reference, define the sets
\begin{align*}
Q_\alpha^>(\bY) = \{u\in E \,|\, p_{\bY}(u) > \alpha\}, &&
Q_\alpha^=(\bY) = \{u\in E \,|\, p_{\bY}(u) = \alpha\}.
\end{align*}
Note that the measurability of these sets is implied by the upper-semicontinuity (and thus measurability) of $p_\bY$. It goes without saying that the quantities $Q_\a$, $Q_\a^>$ and $Q_\a^=$ are law-invariant functionals in that they only depend on the distribution $F_{\bY}$ of a random closed set $\bY$, and, \emph{a fortiori}, on its coverage function $p_\bY$. Therefore, we shall consider them as maps defined on some generic specification of distributions $\M$.

Proposition 2.2.8 in \cite{Molchanov2017} establishes that the symmetric difference in measure
\be{eq:sym diff}
S_{1/2} \colon \mathfrak U \times \mathfrak U \to \R, \qquad S_{1/2}(X,Y)=\tfrac{1}{2}\mu(X\triangle Y)
\ee
is an $\M$-consistent exhaustive scoring function for the Vorob'ev median $Q_{1/2}(\bY)\colon\M\to \mathfrak U$. 
Other Vorob'ev quantiles solve a restricted minimisation problem; see Proposition 4 in \cite{AzzimontiETAL2018}. More precisely, for $\alpha\in[0,1]$, $Q_\alpha = Q_\alpha(\bY)$ it holds that 
\[
\E\big[\tfrac{1}{2}\mu(Q_\alpha \triangle \bY)\big] \le \E\big[\tfrac{1}{2}\mu(M \triangle \bY)\big]
\]
for all measurable sets $M\subseteq E$ such that $\mu(M) = \mu(Q_\alpha)$.
To arrive at a consistent scoring function for a general $\alpha\in\left[0,1\right]$, we first introduce a strict selective $\M$-identification function for $Q_\alpha^=$.
\begin{prop}\label{prop:identifQ=}
For $\alpha\in[0,1]$,
the function $V_\a\colon E\times\mathfrak U\to\R$, $V_\a(u,Y)=\one_Y(u)-\alpha$,
is a strict selective $\M$-identification function for $Q_\alpha^=$. Moreover, $V_\a$ is oriented in the sense that for all $u\in E$ and for all $F\in\M$
\[
\bar V_\a(u,F)
\begin{cases}
>0, & u\in Q_\alpha^>(F) \\
=0, & u \in Q_{\alpha}^=(F) \\
<0, & u \notin Q_\alpha(F).
\end{cases}
\]
\end{prop}
\begin{proof}
The proof follows directly from the definition of $p_{\bY}$.
\end{proof}
This oriented strict $\M$-identification function for $Q_\alpha^=$ turns out to be the main building block in the construction of an exhaustive $\M$-consistent scoring function for $Q_\alpha$. The rationale is akin to the ones presented for the scalar case by \cite{Ziegel_Discussion2016}, \cite{Dawid_Discussion2016} and the multivariate case in \citet[Section 3.2]{FisslerHlavinovaRudloff_RM}. 
\begin{prop}\label{prop:consistency}
Let $\mathfrak{U}_0\colon=\left\{M\in\mathfrak U\,|\, \mu(M)<\infty\right\}$. For any $\alpha\in[0,1]$ with $Q_\alpha(F)\in\mathfrak U_0$ for all $F\in\M$ 
the function $\widetilde S_\a\colon\mathfrak U_0 \times\mathfrak U\to\R$
\be{eq:integrated score}
\widetilde S_\a(X, Y) = -\int_X V_\a(u, Y)\, \mu(\diff u)= \alpha \mu(X) - \mu(Y\cap X) 
\ee
is an $\M$-consistent exhaustive scoring function for $Q_\alpha$. More precisely, it holds that for all $F\in\M$
\begin{multline}\label{eq:argmin}
\argmin_{X\in\mathfrak U} \E_F\left[\widetilde S_\a(X,\bY)\right]\\
= \{X\in \mathfrak U_0\,|\, \exists D\subseteq E \ \text{\rm measurable}\colon
Q^>_\alpha(F) \subseteq D \subseteq Q_\alpha(F) \ \text{\rm and} \  
\mu(X\triangle D)=0\}\,.
\end{multline}
\end{prop}
\begin{proof}
First note that -- if we extend $\widetilde S_\a$ to the family of measurable subsets of $E$ with finite measure -- it holds that for any such $D\subseteq E$ we have $\widetilde S_\a(X,Y) = \widetilde S_\a(D,Y)$ whenever $\mu(X\triangle D)=0$. 
Now, let $X\in \mathfrak U_0$ such that $\mu(X\triangle D)=0$ for some measurable $Q^>_\alpha(F) \subseteq D \subseteq Q_\alpha(F)$. Then, invoking Robbin's Theorem \cite[Theorem 1.5.16]{Molchanov2017}, it holds that for any 
$M\in\mathfrak U_0$ and any $F\in\M$
\begin{align*}
\E_F\left[\widetilde S_\a(M,\bY)-\widetilde S_\a(X,\bY)\right]
&= \E_F\left[\widetilde S_\a(M,\bY)-\widetilde S_\a(D,\bY)\right]\\
&=\int\limits_{D}\E_F\left[V_\a(u,\bY)\right]\mu(\diff u)-\int\limits_M\E_F\left[V_\a(u,\bY)\right]\mu(\diff u)\\
&=\int\limits_{Q_\alpha^>(F)\setminus M}\bar V_\a(u,F)\mu(\diff u)-\int\limits_{M\setminus Q_\alpha(F)}\bar V_\a(u,F)\mu(\diff u)\geq0,
\end{align*}
where the last inequality follows from the orientation of $V_\a$. Moreover, the inequality is strict if and only if $\mu(Q_\alpha^>(F)\setminus M) + \mu(M\setminus Q_\alpha(F))>0$, which establishes equality in \eqref{eq:argmin}. Indeed, for any $M\in\mathfrak U$ with $\mu(Q_\alpha^>(F)\setminus M) + \mu(M\setminus Q_\alpha(F))=0$ choose $D=(M\cap Q_\alpha(F))\cup Q^>_\alpha(F)$. Then $D$ is measurable and it can be easily verified that $Q_\alpha^>(F)\subseteq D\subseteq Q_\alpha(F)$. Moreover, $D\setminus M= Q_\a^>(F)\setminus M$ and $M\setminus D=M\setminus Q_\a(F)$, so that $\mu(M\triangle D)=0$.
\end{proof}
Proposition~\ref{prop:consistency} and in particular the equality in \eqref{eq:argmin} exactly quantify by how much the score $\widetilde S_\a$ fails to be strictly consistent for $Q_\alpha$. 
Moreover, in contrast to the symmetric difference in measure in \eqref{eq:sym diff}, the score $\widetilde S_\a$ in \eqref{eq:integrated score} assumes both negative and positive values in general. Imposing the normalisation condition 
that $S_\a(Y,Y)=0$ for all $Y\in\mathfrak U$, which implies the non-negativity of $S_\a$, the score $\widetilde S_\a$ in \eqref{eq:integrated score} is equivalent to 
\begin{align*}
S_\a(X,Y) = \widetilde S_\a(X,Y) +(1- \alpha) \mu(Y) 
  &= \alpha \mu(X\setminus Y) +(1-\alpha) \mu(Y\setminus X).
\end{align*}
Moreover, one can see that one really retrieves the symmetric difference in measure for $\a=1/2$.
The following theorem states conditions for the \emph{strict} consistency of $S_\a$. In the sequel we denote the closure of any set $M\subseteq E$ with $\cl(M)$ and its interior with $\interior(M)$. 

\begin{thm}\label{thm:mixture repre}
Let $\alpha\in[0,1]$. 
\begin{enumerate}[\rm (i)]
\item
For any $u\in E$ the elementary score $S_{\alpha,u}\colon\mathfrak U\times\mathfrak U\to[0,\infty)$,
\be{eq:elementary}
S_{\alpha,u}(X,Y)=\alpha\one_{X\setminus Y}(u)+(1-\alpha)\one_{Y\setminus X}(u),
\ee
is a non-negative exhaustive $\M$-consistent scoring function for $Q_\alpha$.
\item
Let $\pi$ be a $\sigma$-finite non-negative measure on $E$. Then the map $S_{\alpha,\pi}\colon\mathfrak U\times \mathfrak U\to[0,\infty]$,
\be{eq:mixture}
S_{\alpha,\pi}(X,Y)=\int S_{\alpha,u}(X,Y)\pi(\diff u)
=\alpha \pi(X\setminus Y) +(1-\alpha) \pi(Y\setminus X)
\ee
is a non-negative exhaustive $\M$-consistent scoring function for $Q_\alpha$.
\item
If $\M$ is such that $Q_\a(F)=\cl(Q_\alpha^>(F))$ and $Q_\a(F)=\cl(\interior(Q_\a(F))$ for all $F\in\M$, then $Q_\alpha$ is exhaustively elicitable on $\M$. \\
Moreover, for any $\sigma$-finite positive measure $\pi$ (that is, $\pi$ assigns positive mass to all open non-empty sets) on $E$ such that $\E_F[\pi(\bY)]<\infty$ and $\pi(Q_\a(F))<\infty$ for all $F\in\M$, the restriction of 
$S_{\alpha,\pi}$ defined in \eqref{eq:mixture} to the family $\mathfrak{U}'\colon=\left\{U\in\mathfrak U\,|\,U=\cl(\interior(U))\right\}$ is a strictly $\M$-consistent exhaustive scoring function for $Q_\alpha$.
\end{enumerate}
\end{thm}

To prove this theorem, we will need an auxiliary result that we introduce now.

\begin{lem}
If for two sets $A,B\subseteq E$ it holds that $A=\cl(\interior(A))$ and $B=\cl(\interior(B))$, then $A\triangle B\neq\emptyset$ implies $\interior(A\triangle B)\neq\emptyset$.
\end{lem}
\begin{proof}
Let $A,B$ as above and assume that there is some $x\in A\triangle B$. Without loss of generality assume $x\in A\setminus B$. Since $x\in A$, there is a sequence $(a_n)_{n\in\N}\subseteq \interior(A)$ converging to $x$. Moreover, since $B$ is closed and $x\notin B$, there is some $m\in\N$ such that for all $n\geq m$ it holds that $a_n\notin B$. Thus, for all $n\geq m$ we have $a_n\in \interior(A)\setminus B=\interior(A)\cap B^c\subseteq A\cap B^c$. Since the interior of a set is the union of all of its open subsets,  $\emptyset\neq \interior(A)\cap B^c\subseteq \interior(A\cap B^c)$.
\end{proof}
\begin{proof}[Proof of Theorem~\ref{thm:mixture repre}]
The proof of (i) follows along the lines of the proof of Proposition~\ref{prop:consistency} upon setting $\mu = \delta_u$. Note that with this choice of $\mu$, any set is of finite measure. (ii) is a direct consequence of the nonnegativity and consistency of $S_{\a,u}(X,Y)$. \\
For (iii), let $F\in\M$ such that $\E_F[\pi(\bY)]<\infty$ and note that for any $M\in \mathfrak U'$ with $\pi(M)=\infty$, we have $\bar{S}_{\a,\pi}(M,F)=\infty$. Therefore it suffices to consider $M\in \mathfrak U'$ with $\pi(M)<\infty$ and one can invoke the equality in \eqref{eq:argmin}. If $Q_\alpha(F)$ is the topological closure of $Q_\alpha^>(F)$, then $X = Q_\alpha(F)$ is the only \emph{closed} set such that $Q^>_\alpha(F) \subseteq X \subseteq Q_\alpha(F)$. For any other closed set $M\in\mathfrak U'$ we therefore obtain that $X\triangle M \neq \emptyset$. 
This implies that $\interior(X\triangle M) \neq \emptyset$ and therefore, since $\pi$ is positive, $\pi(X\triangle M)>0$.
\end{proof}
The orientation of the selective identification function $V_\a$ directly implies order-sensi\-tivity in the sense of \cite{Nau1985} or \cite{FisslerZiegel2019} with respect to the partial order induced by the subset relation.
\begin{prop}
Let $\a\in[0,1]$. Then any exhaustive $\M$-consistent scoring function $S_{\a,\pi}$ for $Q_\a$ of the form \eqref{eq:mixture} is order-sensitive. That means for any $F\in\M$ and for any $A,B\in\mathfrak U$ such that $Q_\a(F) \subseteq A\subseteq B$ or $B\subseteq A\subseteq Q_\a(F)$ it holds that $\bar S_{\a,\pi}(A,F) \le \bar S_{\a,\pi}(B,F)$.
\end{prop}

It is worth to explore further connections between mixture representation of consistent scoring functions established for Vorob'ev quantiles in Theorem~\ref{thm:mixture repre} and the corresponding mixture representation in the one-dimensional case, which was introduced and discussed in \cite{EhmETAL2016}.
Indeed, the elementary scores introduced there, 
\[
S^Q_{\alpha,\theta}(x,y)=\left(\one\left\{y<x\right\}-\alpha\right)\left(\one\left\{\theta<x\right\}-\one\left\{(\theta<y\right\}\right), \quad x,y, \theta\in\R
\] 
can be rewritten as $\alpha\one_{[x,\infty)\setminus[y,\infty)}(\theta)+(1-\alpha)\one_{[y,\infty)\setminus[x,\infty)}(\theta)$.
Of course, we can identify the reals $x, y$ with the corresponding sets $X= [x,\infty)$ and $Y = [y,\infty)$, which shows that we end up with the form given in \eqref{eq:elementary}. For further analogy, let $Z$ be a real-valued random variable. This induces a random closed set $\mathbf{Z} = [Z,\infty)$. Then it holds that 
\[
Q_\a(\mathbf{Z}) = [q_\a^-(Z), \infty).
\]
As discussed in Example~\ref{examples} (iv), the elicitability of $q_\alpha^-(Z)$ is equivalent to the exhaustive elicitability of $[q_\alpha^-(Z),\infty)$. One can easily check that for a positive measure $H$ on $\R$, $$S(x,y)=\int S^Q_{\alpha,\theta}(x,y) H(\diff \theta)$$ is a strictly consistent exhaustive scoring function for $[q_\alpha^-(Z),\infty)$ if and only if the closure of $Q_\a^>(\mathbf{Z}) = (q_\a^+(Z), \infty)$, where $q_\a^+(Z)$ is the upper $\a$ quantile of $Z$, corresponds to $Q_\a(\mathbf{Z}) = [q_\a^-(Z), \infty)$. That is, if and only if the $\a$-quantile of $Z$ is unique. This retrieves the first condition in part (iii) of Theorem~\ref{thm:mixture repre}. Note that in the case of a one-dimensional quantile, the second condition is equivalent to $q_\a^-(Z) = q_\a^+(Z)$, too. However, in the case of Vorob'ev quantiles it is more involved and does not follow from the first condition in general. This structural difference also highlights the importance of a thorough framework for dealing with set-valued functionals.

\section{Connections to forecast evaluation in the literature}
\label{subsec:literature}

We would like to close the paper with a comprehensive literature review of different practices of treating forecasts for set-valued functionals. 
We think that these various perspectives illustrate the advantage our unified theoretical framework on set-valued forecast evaluation, with the thorough distinction between a selective and an exhaustive mode, offers. At the same time, these perspectives offer numerous starting points for further research projects to uncover their behaviour in terms of the classification into selectively elicitable functionals, exhaustively elicitable functionals, and functionals failing to be elicitable at all.

\subsection{Statistical forecast evaluation}
\label{subsec:forecast evaluation}

While \cite{LambertETAL2008} only consider real-valued functionals where the distinction between selective and exhaustive scoring functions is superfluous, the influential paper \cite{Gneiting2011} treats functionals as potentially set-valued; cf.\ \cite{LambertShoham2009, BelliniBignozzi2015}.
However, only the concept of selective scoring functions with the corresponding notion of (strict) consistency and elicitability are given. Presumably, the motivation for doing so was induced by the quantile-functional as one of the most prominent examples of a set-valued functional. To the best of our knowledge, forecasts for the quantile are exclusively considered in the selective sense \citep{Koenker2005, Komunjer2005, Gneiting2011b}, in which they are elicitable. The reason for not considering them in the exhaustive sense might lie in the impossibility of establishing corresponding elicitability results, of which the first formal proof---to the best of our knowledge---is given in this paper.

The recent preprint \cite{BrehmerGneiting2020} considers elicitability for the class of predictive intervals and certain specifications thereof through the lens of the selective notion.

\subsection{Statistical theory and risk measurement}

Quantiles and expectiles \citep{NeweyPowell1987} of univariate distributions are well known (selectively) elicitable functionals. In the literature on quantitative risk management, they are also common scalar risk measures. There are different competing attempts to generalise them to a multivariate setting. We refer the reader to two recent and insightful papers and the corresponding references therein: \cite{HamelKostner2018} introduce multivariate quantiles taking the form of convex sets, and \cite{DaouiaPaindaveine2019} introduce hyperplane-valued multivariate $M$-quantiles with a particular focus on hyperplane-valued multivariate expectiles. For both approaches, it remains an intriguing open question whether these functionals are selectively elicitable, exhaustively elicitable or not elicitable at all.

The newly established framework has been applied in \cite{FisslerHlavinovaRudloff_RM}, providing exhaustive elicitability results and selective identifiability results for set-valued systemic risk measures introduced in \cite{FeinsteinRudloffWeber2017}.

\subsection{Spatial statistics}

As mentioned at the beginning of Section~\ref{sec:Vorob'ev Quantiles}, estimating set-valued quantities is a common endeavour in spatial statistics. In that context, forecasts and estimates are commonly considered with what we call an exhaustive angle. Interesting open theoretical questions besides Vorob'ev quantiles are to consider other functionals, notably expectations, of random sets presented in the book \cite{Molchanov2017}.

One area of particular interest in spatial statistics is meteorology and climatology. In these disciplines, forecast evaluation is more commonly known under the term \emph{forecast verification}. We refer the reader to the comprehensive overview paper \cite{DorningerETAL2018}. Besides simply comparing a set-valued forecast and a set-valued observation as outlined above, there are also more involved situations covered. E.g.\ acknowledging the spatio-temporal structure of many processes such as precipitation, one might evaluate probabilistic forecasts for the marginal distributions of the random field of interest at certain grid points, using the \emph{neighbourhood} method (see \cite{DorningerETAL2018} for references). Assessing the entire joint distribution of the random field seems extremely ambitious and we are unaware of any verification method at the moment.

\subsection{Regression and Machine Learning}

Recent literature on isotonic regression embraces the idea of explicitly modelling functionals as set-valued; see \cite{JordanMuehlemannZiegel2019} and \cite{MoeschingDuembgen2019}, where the two papers consider these functionals in the selective sense.

\cite{KivaranovicETAL2019} examine how to obtain prediction intervals with deep neural networks.
In the area of machine learning, the recent paper \cite{GaoETAL2019} considers set-valued regression as well, however, considering finite sets only. The observations (or response variables) $Y_t$ are finite subsets of some label space $S$, which is assumed to be at most countably infinite. Denoting the regressors with $X_t\in\R^p$ then they are interested in finding a function $m\colon \R^p\to \{I\,|\,I\subseteq S, \ |I|<\infty\}$ such that $m(X_t)$ is reasonably close to $Y_t$. However, they do not explicitly specify the loss function they use for the regression problem.
In an orthogonal direction, \cite{ZaheerETAL2017} consider the case of set-valued regressors rather than set-valued responses, which does not lead to the question of an appropriate choice of loss function with set-valued arguments.

\subsection{Philosophy}

Within a more philosophical strand of literature about \emph{credences}, i.e., subjective probabilities of degrees of belief, \cite{Mayo-WilsonWheeler2016} argue that imprecise credences about the probability of a binary event can be represented as subsets of the unit interval $[0,1]$; cf.\ \cite{SeidenfeldETAL2012}. They consider numerical accuracy measures, being functions of the set-valued credence and the binary outcome. In this regard, they consider scoring functions taking sets as arguments. However, this ansatz is distinct from our focus since we consider forecasts for functionals which are inherently set-valued and dispense with a discussion of subjective probabilities, whereas they consider set-valued forecasts for a functional which is actually real-valued, namely the probability of a binary event.

\appendix
\section*{Appendix}

\section{Prediction intervals with a fixed endpoint or midpoint}
\label{subsec:fixed points}
In this subsection, we consider $\alpha$-prediction intervals where one endpoint or the middle of the interval is fixed \textit{a priori} at some point $a\in\R$ or $c\in\R$, respectively. We only address the case of the left endpoint being fixed---the case of the right endpoint follows \textit{mutatis mutandis}. Clearly, if the position of the interval is specified in this sense, reporting such an interval boils down to reporting a one-dimensional quantity, e.g., the other endpoint or the length of the interval.
For a fixed endpoint, however, the existence of such a prediction interval is in general no longer guaranteed (there might not be enough mass on the right of the lower endpoint), and one needs to specify the class of distribution functions more carefully. Therefore, for some fixed $\alpha\in(0,1)$, we introduce $\M_{a,0}=\{F\in\M_0\,|\,F(a-)\leq1-\alpha\}$ for some fixed $a\in\R$. 
 Moreover, let $\M_{a,\a, \text{inc}}$ be the subset of distributions in $\M_{a,0}$ such that $q_{\a+F(a-)}(F)$ is a singleton (with $\Gamma_\a(F)(a)$ being its unique element).
We start with an endpoint specification. 
\begin{prop}\label{prop:fix_left}
\begin{enumerate}[\rm (i)]
\item
On $\M_{a,0} \cap \M_{\text{\rm inc,\,cont}}$ the functional $F\mapsto \Gamma_\a(F)(a)\in [a,\infty]$ is identifiable with a strict identification function
\[
V\colon[a,\infty]\times\R\to\R,\qquad V(x,y)=\one\{y\in [a,x]\}-\alpha.
\]
\item 
On $\M_{a,\a, \text{\rm inc}}$ the functional $F\mapsto \Gamma_\a(F)(a)\in [a,\infty]$ is elicitable. An $\M_{a,0}$-consistent scoring function is given by $S\colon [a,\infty]\times\R\to\R$
\begin{align}\label{eq:I_a_exh}
S(x,y)&=\int\limits_{[x,\infty)}V(z,y)\diff\mu(z)= \a \mu([x,\infty)) -  \mu\big([x,\infty)\cap[y,\infty)\big)\one\{y\ge a\}
\end{align}
where $\mu$ is a non-negative finite measure on $\R$.
If moreover $\mu$ is positive, $S$ is also strictly $\M_{a,\a, \text{\rm inc}}$-consistent.
\end{enumerate}
\end{prop}
\begin{proof}
The proof follows along the lines of the proof of Theorem~\ref{thm:intervals_class}.
\end{proof}
\begin{rem}\label{rem:equivalent score}
A non-negative equivalent version of the score in \eqref{eq:I_a_exh} is given via 
\begin{align}\nonumber
\widetilde S(x,y) &= \one\{y\ge a\}\Big((1-\a)\mu\big([y,\infty)\setminus [x,\infty)\big) + \a\mu\big([x,\infty)\setminus [y,\infty)\big)\Big) + \one\{y<a\}\a \mu\big([x,\infty)\big) \\
&= \one\{y\ge a\}\Big(\big(\one\{y\le x\} - \a\big) \big(h(y) - h(x)\big)\Big) + \one\{y<a\}\a h(x)\,,\label{eq:rem}
\end{align}
where $h\colon(-\infty,\infty]\to\R$ is a decreasing function given by $h(t) = \mu([t,\infty))$ such that $h(+\infty)=0$. Besides the obvious interpretation of the first line of \eqref{eq:rem} in the context of mixture representation, it is remarkable to see the structural similarity to a standard quantile score in the second line of \eqref{eq:rem}, which is of the form $\big(\one\{y\le x\} - \a\big) \big(h(y) - h(x)\big)$. In fact, if the whole support of $F$ lies above $a$, the right endpoint of the resulting interval is the $\a$-quantile. If $F$ assigns positive mass to $(-\infty,a)$, there is a correction term accounting for the fact that the $\a$-quantile would not be sufficient to achieve the required coverage anymore. Note that without the correction term, one is in the setting of Theorem 5 in \cite{Gneiting2011} with $w(y)=\one\{y\ge a\}$. This would correspond to forecasting the $\a$-quantile of $F$ truncated at $a$, i.e., loosely speaking the point under which $\a\cdot100\%$ of the mass above $a$ is located. The aim, however, is to report a point such that $\a\cdot100\%$ of the \emph{whole mass} is between $a$ and the reported point.
\end{rem}

\begin{prop}
For some $m\in\R$ consider the functional $b_m\colon F\mapsto d_\a(F)(m)\in[0,\infty)$, specifying half of the length of the shortest $\a$-prediction interval with midpoint $m$.
Then the following assertions hold:
\begin{enumerate}[\rm (i)]
\item $b_m$ is identifiable on $\M_{\text{\rm inc,\,cont}}$ with a strict identification function  
\[
V\colon[0,\infty)\times\R\to\R,\qquad (x,y)\mapsto V(x,y)=\one\{y\in [m-x,m+x]\}-\alpha.
\]
\item $b_m$ is elicitable on $\M_{\text{\rm inc,\,cont}}$. 
If $\mu$ is a finite non-negative measure on $[0,\infty)$, then $S\colon [0,\infty) \times\R\to\R$
\be{eq:I_c_exh}
S(x,y)=\int\limits_{[0,x)}V(z,y)\diff\mu(z) = \mu\big([0,x) \cap [|y-m|, \infty)\big) - \a\mu\big([0,x)\big)
\ee
is an $\M_{\text{\rm inc,\,cont}}$-consistent scoring function for $b_m$. It is strictly $\M_{\text{\rm inc,\,cont}}$-consistent if $\mu$ is positive. 
\end{enumerate}
\end{prop}
\begin{proof}
The proof follows along the lines of the proof of Theorem~\ref{thm:intervals_class}.
\end{proof}
\begin{rem}
A non-negative equivalent version of the score in \eqref{eq:I_c_exh} is given via 
\begin{align}\nonumber
\widetilde S(x,y) &= \mu\big([0,x) \cap [|y-m|, \infty)\big) - \a\mu\big([0,x)\big) +\a\mu\big([0,|y-m|)\big)  \\
&= \big(\one\{|y-m|\le x\} - \a\big) \big(g(x) - g(|y-m|)\big)\,,\nonumber
\end{align}
where $g\colon[0,\infty)\to\R$ is an increasing function given by $g(t) = \mu([0,t))$ such that $h(0)=0$. Again we see the structural similarity to a standard quantile score in the second line. In particular, we see that $b_m$ corresponds to the $\a$-quantile of the distribution of $|Y-m|$.
\end{rem}

\section{Injectivity Results for Prediction Interval Variants}\label{sec:inject-results-pred}

\newcommand{\B}{\mathcal{B}}

For simplicity, let us consider the class $\M^c \subseteq \M_0$ of probability measures with single-valued quantiles in the range $(0,1)$, i.e., supported on an interval (potentially all of $\R$) and whose CDFs are strictly increasing on that interval.
We first observe that if a functional value $T(F)$ uniquely determines a dense set of quantiles for $F$, then $T$ must be injective.

\begin{lem}\label{lem:functional-injective}
  For some set $W$, let $T:\M^c\to 2^{W}$ and let $Q\subseteq(0,1)$ be dense.
  If for all $F\in\M^c$, the value of $T(F)$ uniquely determines the values of $q_\beta(F)$ for all $\beta\in Q$, then $T$ is injective.
\end{lem}
\begin{proof}
  By definition of $\M^c$, we have $q_\beta(F) = F^{-1}(\beta)$ for all $F\in\M^c$ and $\beta\in(0,1)$.
  As $F$ is continuous and strictly monotone, its inverse is also continuous on $(0,1)$ and strictly monotone.
  Thus, specifying the values of $F^{-1}$ on a dense subset of $(0,1)$ uniquely specifies $F^{-1}$ and thus $F$.  
\end{proof}

For $\alpha\in(0,1]$ we will now define the collection $\C_\a(F)$ of all $\alpha$-prediction sets of $F$, as well as unions of two prediction intervals $\I^2_\a(F)$ and wrapped intervals $\I^w_\a(F)$.
Recall the definition $\bar U := \{(a,b)^\intercal \in \bar \R^2\,|\,a\le b\}$, where $\bar\R\colon =\R\cup\left\{-\infty,\infty\right\}$.
In what follows, we will overload notation and interpret $I\in\bar U$ as a closed interval, so for example if $I=(a,b)^\intercal$ we have $F(I) = F([a,b])$.
\begin{align*}
  &\C_\a\colon \M^c\to 2^{\B(\R)}, \;\; F\mapsto \{B\in \B(\R) : F(B) \geq \a\}\,,\\
  &\I^2_\a\colon \M^c\to 2^{\B(\R)}, \;\; F\mapsto \{(I_1,I_2)\in \bar U^2 : 
    F(I_1\cup I_2)\geq\a\}\,,\\
  &\I^w_\a\colon \M^c\to 2^{\B(\R)}, \;\; F\mapsto \I_\a(F) \cup \{(I_1,I_2)\in\I^2_\a(F) : I_1= (a, \infty)^\intercal,\, I_2=(-\infty, b)^\intercal, a \geq b\}\,.
\end{align*}
We first show injectivity of $\I^2_\a$, and thus $\C_\a$.
We will routinely rely on the bijection between the above  $\I^2_a$ and the functional $\I^2_{=\a} \colon \M^c\to 2^{\B(\R)}$ defined by $\I^2_{=\a}(F) = \{(I_1,I_2)\in\I^2_\a : F(I_1\cup I_2) = \a\}$.   
It is clear that $\I^2_\a(F)$ can be constructed from $\I^2_{=\a}(F)$ and vice versa, by adding or removing nested intervals.
  In particular, $\I^2_\a$ is injective if and only if $\I^2_{=\a}$ is.

\begin{prop}
  \label{prop:double-interval-injective}
  $\I^2_\a$ is injective for all $\a \in (0,1)$.
\end{prop}
\begin{proof}
  We will instead show injectivity of the functional $\I^2_{=\a}$.
  Let $F\in\M^c$, and consider first the case $\alpha\leq 1/2$.
  Given $\I^2_{=\a}(F)$, we will show how to compute the quantiles $q_{k,n} := F^{-1}(k\alpha/2^n)$ for all $k\in\N$, $n\in\N_0$ such that $k\alpha/2^n \in (0,1)$, at which point the result will follow from Lemma~\ref{lem:functional-injective}.
  We first show the result for $k \leq 2^n$; the other values will follow from the observation that $F((q_{k,n},q_{k+2^n,n}])=\alpha$.
  
  As a base case, consider $n=0$ and $k=1$.
  The value $q_{1,0}$ defines the unique interval of the form $(-\infty,q_{1,0}]$ such that $F((-\infty,q_{1,0}])=\alpha$.
  Thus, we may take any $(I_1,I_2)\in\I^2_{=\a}(F)$ with $I_1 = (-\infty,a)^\intercal$ and $I_2 = (a,b)^\intercal$ and set $q_{1,0} = b$.

  Now assume the value of $q_{k,n} = F^{-1}(k\alpha/2^n)$ is known for all $1\leq k\leq 2^n$; we will show how to compute $q_{k,n+1}$ for all $1\leq k \leq 2^{n+1}$.
  We will show that there is a unique triple of intervals of the form $I_1 = (-\infty,x)^\intercal, I_2 = (x,q_{1,n})^\intercal, I_3 =(q_{1,0},y)^\intercal$ such that $(I_1,I_3),(I_2,I_3)\in\I^2_{=\a}(F)$.
  For existence, take $x = q_{1,n+1}$ and $y = q_{2^{n+1}+1,n+1}$.
  For uniqueness, we have $F(I_1)+F(I_2)=\alpha/2^n$ by definition of $q_{1,n}$, and $F(I_1)+F(I_3)=F(I_2)+F(I_3)=\alpha$ by definition of $\I^2_{=\a}$, from which we conclude $F(I_1)=F(I_2)=\alpha/2^{n+1}$.
  Thus, we must have $x=q_{1,n+1}$, and the value of $y$ follows.
  This triple therefore uniquely determines $q_{1,n+1}$.
  To recover the other values of $k$, observe that for all $k\leq 2^n$ we trivially have $q_{2k,n+1} = q_{k,n}$, and the unique $I'_k = (q_{k,n},z)^\intercal$ with $(I'_k,I_3)\in\I^2_{=\a}(F)$ satisfies $z=q_{2k+1,n+1}$.
  
  When $\alpha > 1/2$, we may proceed with the previous construction replacing $\alpha$ with $\beta = 1-\alpha$, as follows.
  We now let $q_{k,n} := F^{-1}(k\beta/2^n)$ for all $k,n\in\N$ such that $k\beta/2^n \in (0,1)$.
  Again, we first show how to compute these values for $1\leq k \leq 2^n$, as the other values follow from the observation $F((-\infty,q_{k,n}]\cup[q_{k+2^n,n},\infty))=\alpha$.
  
  The base of the induction defines $q_{1,0}=x$ from the unique interval $I=(x,\infty)^\intercal$ such that $F(I)=\alpha$.
  To induct, we again ask for intervals $I_1 = (-\infty,x)^\intercal, I_2 = (x,q_{1,n})^\intercal, I_3 =(q_{1,0},y)^\intercal$ such that $(I_1,I_3),(I_2,I_3)\in\I^2_{=\a}(F)$, with the same argument as before replacing $\alpha$ with $\beta$.
  Finally, for all $1\leq k\leq 2^n$ we again have $q_{2k,n+1} = q_{k,n}$, and the unique $I'_k = (q_{k,n},x)^\intercal$ with $(I'_k,I_3)\in\I^2_{=\a}(F)$ gives $x=q_{2k+1,n+1}$.
\end{proof}

\begin{cor}
  $\C_\a$ is injective for all $\a\in(0,1)$.
\end{cor}

\begin{prop}
  \label{prop:I-wrapped-injective}
  Let $\a\in(0,1)$.  $\I^w_\a$ is injective if and only if $\alpha$ is irrational.
\end{prop}

\begin{proof}
  First, consider irrational $\a\in(0,1)$.
  As in Proposition~\ref{prop:double-interval-injective}, we will instead show injectivity of   $\I^w_{=\a} \colon \M^c\to 2^{\B(\R)}, \;\; F\mapsto \I_{=\a}(F) \cup \{(I_1,I_2)\in\I^2_{=\a}(F) : I_1=(-\infty, b)^\intercal,\ I_2= (a, \infty)^\intercal, b \leq a\}$, where naturally $\I_{=\a}(F) = \{I\in\I_a(F) : F(I) = \a\}$.

  For any $q\in(0,1)$, suppose we have established $x=F^{-1}(q)$.
  If $q + \a \leq 1$, then by taking $I \in \I^w_{=\a}(F)$ with $I = (x,y)^\intercal$ for some $y\in\R$, we may conclude $F(y) = q + \a$.
  If $q + \a > 1$, then by taking $(I_1,I_2) \in \I^w_{=\a}(F)$ with $I_1 = (x,\infty)^\intercal$, $I_2=(-\infty,y)$ for some $y\in\R$, we may conclude $F(y) = q + \a - 1 \in (0,1)$.
  In both cases, therefore, the values of $F^{-1}(q)$ and $\I^w_{=\a}(F)$ together determine the value of $F^{-1}((q + \a) \% 1)$, where $a \% b := a - b \lfloor a/b \rfloor$ is the real modulus (see Proposition~\ref{prop:Ialpha-not-injective}).

  Let $x_1\in\R$ be defined by $(-\infty,x_1)^\intercal \in \I^w_{=\a}(F)$, so that $F(x_1) = \a$.
  Proceeding as above, for all $k\in\N$, we may determine the values $x_k\in\R$ by $F(x_k) = k\alpha \% 1$.
  As $\a$ is irrational, the set $\{k\a \% 1: k\in\N\}$ is dense in $(0,1)$, concluding the proof.
  
  Now consider rational $\a\in(0,1)$, so that $\alpha=k/n$ for some $k,n\in\N$, $k < n$.
  For any $0 \le b < 1/(2\pi n)$ define the cumulative distribution function $F_b$ by $F_b(y) = y + b \sin(2\pi ny)$ for $y\in [0,1]$.
  Then for all $y\in [0,1-\a]$, we have
  \[ F_b(y+\a) - F_b(y) = \a + b \sin(2\pi n(y+k/n)) - b \sin(2\pi ny) = \a\,.\]
  For $1-\a < y \leq 1$, let $z = y+\a-1$.
  Then
  \begin{align*}
    F_b&(1)-F_b(y) + F_b(z)-F_b(0)
    \\
       &= (1-y) + b \Bigl( \sin(2\pi n) - \sin(2\pi ny) \Bigr)
         + z + b \Bigl( \sin(2\pi n z) - \sin(0) \Bigr)
    \\
       &= (1-y+z) + b \Bigl( \sin(2\pi n z) -\sin(2\pi ny) \Bigr)
    \\
       &= \a + b \Bigl( \sin(2\pi n (1+z)) -\sin(2\pi ny) \Bigr)
    \\
       &= \a + b \Bigl( \sin(2\pi n (y + k/n)) -\sin(2\pi ny) \Bigr)
    \\
       &= \a~.
  \end{align*}
  We conclude, for all $0 \le b < 1/(2\pi n)$, that
  \begin{align*}
    \I^w_{=\a}(F_b)
    =
    &\{(y,y+\alpha)^\intercal:y\in[0,1-\alpha]\} \\
    &\cup \{(y,\alpha)^\intercal:y\leq 0\} \\
    &\cup \{(1-\alpha,y)^\intercal:y\geq 1\} \\
    &\cup \{((y,\infty)^\intercal,(-\infty,y+\alpha-1)^\intercal) : y\in[1-\alpha,1]\}\,.
  \end{align*}
  Since $\I^w_\a$ is in turn determined by $\I^w_{=\a}$, it also fails to be injective for rational $\a$.
\end{proof}

\section{Omitted Technical Proofs}\label{subsec:proofs}

\begin{proof}[Proof of Theorem \ref{thm:intervals_class}(iii)]
Let $\mu$ be positive, $F\in\M_{\a,\text{\rm inc}}$, $A^* = \I_\a(F) \in\mathcal U^*$ and $A\in\mathcal U^*$ such that $A\neq A^*$. Let $\gamma\colon [-\infty, b]\to(-\infty, \infty]$ and $\gamma^*\colon [-\infty, b^*]\to(-\infty, \infty]$, $b,b^*\in\R$, such that $A = \epi \gamma$ and $A^* = \epi \gamma^*$. We first show that $\interior((A\triangle A^*)\cap U) \neq \emptyset$. 
If $b\neq b^*$, it is obvious that $\interior((A\triangle A^*)\cap U) \neq \emptyset$, due to the fact that $\gamma$ and $\gamma^*$ are left-continuous and since $\gamma$ and $\gamma^*$ assume infinity at most at $b$ and $b^*$, respectively. 
Now assume that $b=b^*$.
If $\gamma(-\infty) \neq \gamma^*(-\infty)$, then the right-continuity at $-\infty$ implies that there is some $a\in \R$ such that $\gamma(x)\neq \gamma^*(x)$ for all $x\in(-\infty,a)$. Hence $\interior((A\triangle A^*)\cap U) \neq \emptyset$.
Finally, if there is some $a\in(-\infty, b]$ such that $\gamma(a) \neq \gamma^*(a)$, then the left-continuity implies that there is some $\varepsilon>0$ such that $\gamma \neq \gamma^*$ on $(a-\varepsilon, a]$. Hence, again $\interior((A\triangle A^*)\cap U) \neq \emptyset$. \\
Since $\bar V(x,F)<0$ for $x\notin A^*$, we obtain a strict inequality in \eqref{eq:proof} if $\interior((A\setminus A^*)\cap U)\neq\emptyset$.
Otherwise, consider any $x\in\interior((A^*\setminus A)\cap U)\neq\emptyset$. If $\bar V(x,F)>0$ for all such $x$, we are done. Suppose there is some $(x_1,x_2)^\intercal\in\interior((A^*\setminus A)\cap U)$ such that $\bar V(x_1,x_2,F)=0$. We show that there exists some $\delta>0$ such that $\bar V(\cdot, F)>0$ on the open rectangle $(x_1-\delta, x_1)\times (\gamma^*(x_1), x_2)$. 
First note that there is some $\delta>0$ such that $(x_1-\delta, x_1)\times (\gamma^*(x_1), x_2)\subset \interior((A^*\setminus A)\cap U)$, since the latter set is open and since both $A$ and $A^*$ are upper sets with ordering cone $(-\infty, 0]\times[0,\infty)$. 
Therefore, it is sufficient to show that for all $z_1\in(x_1-\delta, x_1)$ it holds that $\bar V(z_1,\gamma^*(x_1), F) = F([z_1, \gamma^*(x_1)])-\a>0$.
If $\bar V(x_1,x_2,F)=0$ that means that $F((\gamma^*(x_1), x_2])=0$. Therefore, $[\gamma^*(x_1), x_2]\subseteq q_{\a+F(x_1-)}(F)$. Since $F\in\M_{\a,\text{inc}}$, that means that $q_{F(x_1-)}(F) = \{x_1\}$. It is straight forward to see that for all $z_1\in(x_1-\delta, x_1)$ we have that $F(z_1-) < F(x_1-)$. Hence, 
\[
F([z_1, \gamma^*(x_1)])-\a = F(\gamma^*(x_1)) - F(z_1-)-\a > F(\gamma^*(x_1)) - F(x_1-)-\a\ge0\,.
\]
This yields the claim.
\end{proof}

\begin{proof}[Proof of Proposition~\ref{prop:l}]
Part (i) follows from the fact that $V_l$ is a strict identification function for $l$ and from the strict monotonicity and continuity of $F\in \M_l\cap \M_{\text{\rm inc,\,cont}}$.\\
For (ii) suppose $S$ is a strictly $\M_l\cap \M_{\text{\rm inc,\,cont}}$-consistent scoring function for $T_l$ such that $\bar S(\cdot, F)$ is twice differentiable on $\R\times(0,\infty)$ for all $F\in\M_l\cap \M_{\text{\rm inc,\,cont}}$. Assumptions (a), (b) and (c) and a slight adaptation of \citet[Theorem 3.2, Corollary 3.3]{FisslerZiegel2016} imply the existence of a function $h\colon \R\times(0,\infty)\to\R^{2\times2}$ with differentiable components $h_{ij}$ such that 
\[
\nabla_x \bar{S}(x,F)=h(x)\bar{V}(x,F), \quad \text{for all }x\in \R\times(0,\infty), \ F\in \M_l\cap \M_{\text{\rm inc,\,cont}},
\]
where $V$ is the strict $\M_l$-identification function from part (i).
For any $F\in\M_l$ and $x\in\R\times(0,\infty)$, the Hessian $\nabla^2_x \bar{S}(x,F)$ must be symmetric and, for $x=T_l(F)$, it must be positive semidefinite.
We obtain 
\[
\partial_i\bar{S}(x,F)=h_{i1}(x)\bar{V}_l(x_1,F)+h_{i2}(x)(F(x_1+2x_2)-F(x_1)-\alpha)
\]
for $i=1,2$. From the symmetry of the Hessian it follows that
\begin{align*}
&\partial_2 h_{11}(x)\bar{V}_l(x_1,F)+\partial_2 h_{12}(x)(F(x_1+2x_2)-F(x_1)-\alpha)+2h_{12}(x)f(x_1+2x_2)\\
&\quad =\partial_1h_{21}(x)\bar{V}_l(x,F)+h_{21}(x)\bar{V}_l'(x_1,F)+\partial_1 h_{22}(x)(F(x_1+2x_2)-F(x_1)-\alpha)\\
&\qquad+h_{22}(x)(f(x_1+2x_2)-f(x_1)).
\end{align*}
At $x=T_l(F)$ this yields
\begin{multline*}
2h_{12}(T_l(F))f(l(F)+2b_l(F))\\
=h_{21}(T_l(F))\bar{V}_l'(l(F),F)+h_{22}(T_l(F))(f(l(F)+2b_l(F))-f(l(F))).
\end{multline*} 
The existence of $F_1,\, F_2\in\M_l\cap \M_{\text{\rm inc,\,cont}}$ as in assumption (d) for any $(l^*,b^*)^\intercal\in\R\times(0,\infty)$ as well as the surjectivity of $T_l$ implicitly implied via (a) and (b) yield that $h_{22}\equiv0$. 
Furthermore, the existence of $F_3\in\M_l\cap \M_{\text{\rm inc,\,cont}}$ as in assumption (d) for any $(l^*,b^*)^\intercal\in\R\times(0,\infty)$ together with the surjectivity of $T$ implies that $h_{12}\equiv0$. Finally, the existence of $F_4\in\M_l\cap \M_{\text{\rm inc,\,cont}}$ as assumed implies $h_{21}\equiv 0$.\\
Now, for any $x\in\R\times(0,\infty)$ let $F\in\M_l\cap \M_{\text{\rm inc,\,cont}}$ be such that $l(F)\neq x_1$. Since $V_l$ is a strict identification function for $l$, $\bar{V}_l(x_1,F)\neq0$ and we obtain $\partial_2 h_{11}(x)=0$. Therefore there is a function $g\colon\R\to\R$ such that $h_{11}(x)=g(x_1)$ for all $x\in\R\times(0,\infty)$. In summary, we obtain that 
\[
\nabla_x \bar S(x,F) = \begin{pmatrix}
g(x_1) \bar V_l(x_1,F) \\0
\end{pmatrix},
\]
which implies that $\bar{S}(\cdot,F)$ is constant in $x_2$ and $S$ cannot be strictly consistent for $T_l$. 
\end{proof}
\begin{proof}[Proof of Proposition~\ref{prop:m}]
Part (i) follows from the fact that $V_m$ is a strict $\M$-identification function for $m$ and from the strict monotonicity and continuity of $F \in \M\cap \M_{\text{\rm inc,\,cont}}$.\\
For (ii) suppose $S$ is a strictly $\M\cap \M_{\text{\rm inc,\,cont}}$-consistent scoring function for $T_m$ such that $\bar S(\cdot, F)$ is twice differentiable on $\R\times(0,\infty)$ for all $F\in\M\cap \M_{\text{\rm inc,\,cont}}$. Using exactly the same arguments as in the proof of Proposition~\ref{prop:l},
we can derive the existence of a function $h\colon\R\times(0,\infty)\to\R^{2\times 2}$ with differentiable components $h_{ij}$ such that 
\[
\nabla_x \bar{S}(x,F)=h(x)\bar{V}(x,F), \quad \text{for all }x\in \R\times(0,\infty), \ F\in\M\cap \M_{\text{\rm inc,\,cont}},
\]
where $V$ is the strict $\M\cap \M_{\text{\rm inc,\,cont}}$-identification function from part (i). Again, for any $x\in\R\times(0,\infty)$ and $F\in\M\cap \M_{\text{\rm inc,\,cont}}$, the Hessian $\nabla_x^2 \bar S(x,F)$ must be symmetric and, for $x=T_m(F)$, it must be positive semidefinite.
We obtain that 
\[
\partial_i\bar{S}(x,F)=h_{i1}(x)\bar{V}_l(x_1,F)+h_{i2}(x)(F(x_1+x_2)-F(x_1-x_2)-\alpha)
\]
for $i=1,2$. From the symmetry of the Hessian it follows that
\begin{align*}
&\partial_2 h_{11}(x)\bar{V}_m(x_1,F)+\partial_2 h_{12}(x)(F(x_1+x_2)-F(x_1-x_2)-\alpha) \\
&\qquad +h_{12}(x)(f(x_1+x_2)+f(x_1-x_2))\\
&= \partial_1 h_{21}(x)\bar{V}_m(x_1,F)+h_{21}(x)\bar{V}'_m(x_1,F) +\partial_1 h_{22}(x)(F(x_1+x_2)-F(x_1-x_2)-\alpha)\\
&\qquad +h_{22}(x)(f(x_1+x_2)-f(x_1-x_2)).
\end{align*}
At $x=T_m(F)$ this yields
\begin{multline*}
 h_{12}(T_m(F))(f(m(F)+b_m(F))+f(m(F)-b_m(F))) \\
= h_{21}(T_m(F))\bar{V}'_m(m(F),F)+h_{22}(T_m(F))(f(m(F)+b_m(F))-f(m(F)-b_m(F)))
\end{multline*} 
The existence of $F_1,\, F_2$ as in assumption (d) for any $(m^*,b^*)^\intercal\in\R\times(0,\infty)$ and surjectivity of $T_m$ implicitly given via (a) and (b) imply that $h_{22}\equiv0$. Furthermore, the existence of $F_3$ as in assumption (d) together with the surjectivity of $T_m$ implies that $h_{12}\equiv0$. Finally, the existence of $F_4$ as assumed implies $h_{21}\equiv 0$.
The rest of the argument follows as in the proof of Proposition~\ref{prop:l}.
\end{proof}
\begin{proof}[Proof of Lemma~\ref{lem:SI existence}]
 For $\alpha=1$, please note that $SI_1(F) = \{(\essinf(F), \esssup(F))^\intercal\}$. Now let $\a\in(0,1)$.
First note that $SI_\a(F)\neq\emptyset$ if and only if the function $h(x):=\Gamma_\a(F)(x)-x$ attains its infimum over the interval $P:=\{x\in\R\,|\,F(x-)\leq 1-\alpha\}$ where we note that $P$ is closed, bounded from above and unbounded from below.\\
Assume that $\Gamma_\a(F)$ is continuous. Then $h$ is continuous, too. 
Since $\Gamma_\a(F)(x)\geq x$, $m:=\inf_{x\in P}h(x)\geq0$. 
The tightness of $F$ implies that $m<\infty$.
From the definition of the infimum, there is a sequence $(x_n)_{n\in\N}\subseteq P$ with $h(x_n)\to m$. If this sequence is bounded from below, there is a convergent subsequence $(x_{n_k})_{k\in\N}$ with limit $x\in P$ and the continuity of $h$ implies that $h(x_{n_k})\to h(x)$, thus the infimum is attained.  
If $(x_n)_{n\in\N}$ is not bounded from below, there is a divergent subsequence $(x_{n_l})_{l\in\N}$. But then Lemma~\ref{lem:Gamma} (iv) implies that $h(x_{n_l})\to\infty\neq m$ which is a contradiction.\\
Finally assume that $\Gamma_\a(F)$ fails to be right-continuous such that $h$ is also not continuous. $h$ is discontinuous at $x$ if and only if $\Gamma_\a(F)$ jumps at $x$. Jumps of $\Gamma_\a(F)$ can be caused by two situations, namely if $F$ has jumps or if $F$ has flat spots. Both of them can occur at most countably many times (see e.g.\ Theorem 2.1 in \cite{Shorack2006}) which means that $\Gamma_\a(F)$ can have at most countably many jumps. Let $I = \{1, 2, \ldots, n_0\}$ for some $n_0\in\N$ or $I=\N$ be some index set  and let $(a_i)_{i\in I}$ be the collection of points where $\Gamma_\a(F)$ jumps and let $(j_i)_{i\in I}$ be the corresponding jump sizes.
For all $i\in I$ and any $\eps\in(0,j_i/2]$ it holds that $\Gamma_\a(F)(a_i+\eps)\geq\Gamma_\a(F)(a_i)+j_i$ and consequently that $h(a_i+\eps)>h(a_i)$. Thus if $h$ attains its minimum, it is not in any of the intervals $(a_i,a_i+j_i/2)$, $i\in I$. Now define the sequence of functions $h_i$ in the following way: Set $h_0:=h$. For any $i\in I$, if $h_{i-1}$ is continuous at $a_i$, set $h_i=h_{i-1}$, else 
\[
h_i(x)=
\begin{cases}
u_ix+v_i, & x\in(a_i,a_i+j_i/2) \\
h(x), &\text{otherwise},
\end{cases}
\]
where $u_i=2(h(a_i+j_i/2)-h(a_i))/j_i$ and $v_i=h(a_i)-u_ia_i$. It can easily be verified that $h_i$ is continuous on $[a_i,a_i+j_i/2)$ and moreover that $h_i(x)>h_i(a_i)$ for all $x\in[a_i,a_i+j_i/2)$. Therefore if $h_{i-1}$ attains its infimum at $x^*\in\R$, so does $h_i$ and $h_{i-1}(x^*)=h_i(x^*)$. The pointwise limiting function $h^*$ of $(h_i)_{i\in\N}$ is a continuous function that, by an earlier argument, attains its infimum over $P$. By the construction of the functions $h_i$, $h$ also attains its infimum over $P$, at the same point as $h^*$.
\end{proof}

\section*{Acknowledgement}
We would like to express our sincere gratitude to Tilmann Gneiting and Johanna Ziegel for insightful discussions about the topic, to Alexander Jordan who helped to coin the terminology of \emph{exhaustive} versus \emph{selective} elicitability in a joint discussion, to Dario Azzimonti, Zied Ben Bouall\`{e}gue, Seamus Bradley, Jonas Brehmer, Timo Dimitriadis, David Ginsbourger, Claudio Heinrich, Kory Johnson, Ilya Molchanov, Ruodu Wang, and Dominic Wrazidlo for valuable discussions and helpful references, and to Yuan Li for a careful proofreading of an earlier version of this paper.

Tobias Fissler gratefully acknowledges financial support from Imperial College London via his Chapman Fellowship during which main parts of the project have jointly been developed.
This work was supported in part by the U.S.\ National Science Foundation under Grant No.\ 1657598.

\bibliographystyle{apacite}

\end{document}